\documentclass[12pt,a4paper]{amsart}

\usepackage[T1]{fontenc}
\usepackage{amssymb,amsmath,latexsym,amsthm,paralist,a4,mathrsfs,dsfont}
\usepackage{tikz-cd}
\usetikzlibrary{decorations.markings,decorations.pathmorphing}

\makeatletter
\tikzcdset{
  open/.code     = {\tikzcdset{hook, circled};},
  closed/.code   = {\tikzcdset{hook, slashed};},
  open'/.code    = {\tikzcdset{hook', circled};},
  closed'/.code  = {\tikzcdset{hook', slashed};},
 circled/.code  = {\tikzcdset{markwith = {\draw (0,0) circle (.375ex);}};},
 slashed/.code  = {\tikzcdset{markwith = {\draw[-] (-.4ex,-.4ex) -- (.4ex,.4ex);}};},
  markwith/.code ={
    \pgfutil@ifundefined%
    {tikz@library@decorations.markings@loaded}%
    {\pgfutil@packageerror{tikz-cd}{You need to say %
      \string\usetikzlibrary{decorations.markings} to use arrows with markings}{}}{}%
    \pgfkeysalso{/tikz/postaction = {
      /tikz/decorate,
      /tikz/decoration={markings, mark = at position 0.5 with {#1}}}
    }
  },
}
\makeatother

\usepackage[headings]{fullpage}
\usetikzlibrary{babel}
\usetikzlibrary{positioning,calc}
\usepackage{varioref}
\usepackage[pdfpagelabels, pdftex]{hyperref}
\hypersetup{
  pdftitle={The specializing base change theorem},
  pdfauthor={Katharina H\"{u}bner},
  pdfsubject={},
  pdfkeywords={},
  colorlinks=true,    
  linkcolor=red,     
  citecolor=blue,     
  filecolor=blue,      
  urlcolor=blue,       
  breaklinks=true,
  bookmarksopen=true,
  bookmarksnumbered=true,
  pdfpagemode=UseOutlines,
  plainpages=false,
  unicode=true
  }
\usepackage[capitalise]{cleveref}
\usepackage[inline]{showlabels}
\usepackage{changes}  
\usepackage{bbold} 
\newcommand{\mapsfrom}{\mathrel{\reflectbox{\ensuremath{\mapsto}}}}

\newcommand{\id}{\mathrm{id}}

\newcommand{\m}{\mathfrak{m}}
\newcommand{\p}{\mathfrak{p}}

\newcommand{\Z}{\mathds Z}
\newcommand{\N}{\mathds N}

\newcommand{\et}{\textit{\'et}}

\newcommand{\A}{{\mathbb A}}

\newcommand{\Shv}{\mathit{Shv}}

\newcommand{\disc}{\mathit{disc}}

\newcommand{\cF}{{\mathscr F}}
\newcommand{\cG}{{\mathscr G}}

\newcommand{\cO}{{\mathscr O}}

\newcommand{\cS}{{\mathscr S}}

\newcommand{\cU}{{\mathscr U}}

\newcommand{\cX}{{\mathscr X}}
\newcommand{\cY}{{\mathscr Y}}
\newcommand{\cZ}{{\mathscr Z}}

\newcommand{\liso}{\mathrel{\hbox{$\longrightarrow$} \kern-2.4ex\lower-1ex\hbox{$\scriptstyle\sim$}\kern1.7ex}}

\newcommand{\set}{\textit{s\'et}}
\newcommand{\RZ}{\mathrm{RZ}}

\newcommand{\triv}{\mathrm{triv}}

\makeatletter
\newcommand{\zerounderset}[3][\mathord]{%
  #1{\vtop{
    \let\\\cr
    \baselineskip\z@skip\lineskip.25ex
    \ialign{\hidewidth$##$\hidewidth\crcr
      \omit$#3$\cr
      #2\crcr
    }%
  }}%
}
\makeatother

\newtheoremstyle{alexthm}
  {}
  {}
  {\sl }
  {}
  {\bf}
  {.}
  {.5em}
  {}
\theoremstyle{alexthm}

\newtheorem{theorem}{Theorem}[section]
\newtheorem*{theorem*}{Theorem}
\newtheorem{corollary}[theorem]{Corollary}
\newtheorem{proposition}[theorem]{Proposition}
\newtheorem{lemma}[theorem]{Lemma}
\newtheorem*{lemma*}{Lemma}

\newtheoremstyle{alexdef}
  {}
  {}
  {\rm }
  {}
  {\bf}
  {.}
  {.5em}
  {}
\theoremstyle{alexdef}
\newtheorem*{example*}{Example}
\newtheorem{example}[theorem]{Example}

\newtheorem{remark}[theorem]{Remark}

\newtheorem{definition}[theorem]{Definition}

\DeclareMathOperator{\Spec}{\mathrm{Spec}}
\DeclareMathOperator{\Spv}{\mathrm{Spv}}
\DeclareMathOperator{\Spa}{\mathrm{Spa}}

\DeclareMathOperator{\supp}{\mathrm{supp}}

\DeclareMathOperator{\ch}{char}

\DeclareMathOperator*{\colim}{colim}

\definecolor{darklimegreen}{RGB}{31,142,8}

\begin{document}

\hfuzz=4pt
\title{Tame proper base change for discretely ringed adic spaces}
\author{Katharina H\"{u}bner}
\email{huebner@math.uni-frankfurt.de}
\date{\today}
\address{Institut für Mathematik, Goethe Univeristät Frankfurt}
\thanks{This research is partly supported by ERC Consolidator Grant 770922 - BirNonArchGeom.
 Moreover the author acknowledges support by Deutsche Forschungsgemeinschaft  (DFG) through the Collaborative Research Centre TRR 326 "Geometry and Arithmetic of Uniformized Structures", project number 444845124.}

\begin{abstract}
 We consider a proper morphism $X \to S$ and a locally closed immersion $S' \to S$ of discretely ringed adic spaces and prove proper base change for the tame topology in this setting.
 More precisely, we show for an abelian $p$-torsion sheaf ($p = \ch^+(S)$) on the tame site of~$X$ that the base change homomorphism for the derived pushforward along $X \to S$ with the pullback along~$S' \to S$ is an isomorphism.
\end{abstract}

\maketitle
\tableofcontents

\section{Introduction}

We consider a cartesian diagram
\[
 \begin{tikzcd}
  \cX'	\ar[r,"\iota_\cX"]	\ar[d,"f'"']	& \cX	\ar[d,"f"]	\\
  \cS'	\ar[r,"\iota_\cS"']					& \cS
 \end{tikzcd}
\]
of adic spaces.
By adjunction we always get a base change homomorphism
\[
 \iota_\cS^*Rf_*\cF \longrightarrow Rf'_*\iota_\cX^*\cF
\]
for any sheaf of torsion abelian groups~$\cF$ on the étale site of~$\cX$.
If~$f$ is proper, we expect this base change homomorphism to be an isomorphism.
For analytic adic spaces with the usual finiteness assumptions this was shown in Huber's book \cite{Hu96}, \S~4.4, at least for torsion sheaves whose torsion is invertible on~$\cX$.
Especially for~$p$-torsion coefficients the theorem was largely extended, see \cite{Mann22}, Theorem~1.2.4.

We are intersested in a similar base change theorem for the tame site.
More precisely, we are going to prove the following result.

\begin{theorem} \label{bc_p_torsion}
 Let $f \colon \cX \to \cS$ be a proper morphism of discretely ringed pseudoadic spaces of specialization characteristic $\ch^+(\cS) = p > 0$.
 We consider the cartesian square
 \[
  \begin{tikzcd}
   \cX'	\ar[r,"\iota_\cX"]	\ar[d,"f'"']	& \cX	\ar[d,"f"]	\\
   \cS'	\ar[r,"\iota_\cS"']					& \cS
  \end{tikzcd}
 \]
 for a locally closed immersion~$\iota_\cS$.
 Then for every $p$-torsion sheaf~$\cF$ on the tame site~$\cX_t$, the base change homomorphism
 \[
  \iota_\cS^* Rf_* \cF \longrightarrow Rf'_*\iota_\cX^*\cF
 \]
 is an isomorphism.
\end{theorem}

%

Let us first explain why we cannot expect the ususal statement in full generality, whithout restricting to locally closed immersions.
We consider the following example.
Let $(K,K^+)$ be a tamely closed valued field and $(L/L^+)/(K,K^+)$ a finite wildly ramified Galois extension with Galois group~$G$.
We consider the cartesian square
\[
 \begin{tikzcd}
  \prod_G (L,L^+) = \Spa(L \otimes_K L,L^+ \otimes_{K^+} L^+)	\ar[r]	\ar[d]	& \Spa(L,L^+)	\ar[d]	\\
  \Spa(L,L^+)													\ar[r]			& \Spa(K,K^+).
 \end{tikzcd}
\]
If the base change theorem held for this cartesian square, there would have to be an isomorphism
\[
 \cF(\Spa(L,L^+)) \overset{\sim}{\longrightarrow} \bigoplus_G \cF(\Spa(L,L^+))
\]
for any $p$-torsion sheaf on~$\cX_t$.
This is clearly false.
We could say that the reason for this failure is the cancellation of wild ramification.
In the statement of the above theorem this cannot happen since $\cS' \to \cS$ is an immersion, which has trivial residue field extensions.
It is conceivable to prove a slightly more general statement by imposing some tameness condition on $\cS' \to \cS$ (which a locally closed immersion should certainly satisfy).
However, it seems that such a condition would be quite restrictive and already exclude examples like $\A^1_{(K,K^+)} \to \Spa(K,K^+)$ for some non-algebraically closed valued field $(K,K^+)$.

Let us now explain the strategy for the proof of \cref{bc_p_torsion}.
In fact we prove a base change theorem for the strongly étale topology on~$\cX$ and in the end deduce the result for the tame site using that strongly étale and tame cohomology coincides for $p$-torsion sheaves.
In order to prove the strongly étale base change theorem we study proper adic spaces over a base space of the form $\cS = \Spa(R,R^+)$ for a Huber pair $(R,R^+)$ equipped with the discrete topology.
We prove in \cref{characterisation_proper_over_S} that proper adic spaces over~$\cS$ are always of the form $\cX = \Spa(X,R^+)$ for a scheme~$X$ that is proper over~$R$.
The strongly étale cohomology of~$\cX$ can be computed as a colimit over all $X$-modifications of $\Spec R^+$:
For all sheaves~$\cF$ on~$\cX_\set$, we have
\[
 H^i(\cX_\set,\cF) = \colim_Y H^i(Y_\et,\varphi_{Y,*} \cF),
\]
where the colimit runs over all $X$-modifications of $\Spec R^+$ and $\varphi_Y \colon \cX \to Y$ is the center map (see \cref{cohomology_models}).
The identification of the cohomology of~$\cX$ as a colimit over all modifications is not so straightforward.
This has to do with the fact that as a topological space,~$\cX$ does not coincide with the limit over all $X$-modifications~$Y$ of $\Spec R^+$.
Instead, the limit over all modifications is naturally a subspace of~$\cX$, the space of all Riemann Zariski points~$\cX_\RZ$ (see \cite{Tem11}).
However, the subspace~$\cX_\RZ$ can be thought of as a deformation retract of~$\cX$.
At the level of strongly étale topoi this is reflected by the following diagram (see \cref{comparison_RZ})
\[
 \begin{tikzcd}
  \cX_{\RZ,\set}	\ar[r,"\iota"]	\ar[rr,bend right,"\id"']	& \cX_\set	\ar[r,"\pi"]	& \cX_{\RZ,\set}.
 \end{tikzcd}
\]
The projection $\cX_\set \to \cX_{\RZ,\set}$ is easy to define whereas the embedding $\iota \colon \cX_{\RZ,\set} \to \cX_\set$ requires some work.
The main result of \cref{comparison_RZ} is an isomorphism
\[
 H^i(\cX_\set,\cF) \cong H^i(\cX_{\RZ,\set},\iota^*\cF)
\]
and the right hand side naturally identifies with the colimit over all $X$-modifications~$Y$ of $\Spec R^+$ of $H^i(Y_\et,\varphi_{Y,*}\iota^*\cF)$ (see \cref{sect_specialization}).
The idea now is to apply the proper base change theorem for the étale site of a scheme to these cohomology groups for all modifications~$Y$.

In the first four sections we work with arbitrary adic spaces.
Only starting in \cref{sect_specialization} we restrict our attention to discretely ringed adic spaces.
In \cref{sect_specialization} this is not yet crucial but makes the exposition easier.
The structure of proper adic spaces described in \cref{sect_proper} is much more complicated for arbitrary adic spaces as compared to discretely ringed ones.
In the future we aim for a more general tame base change result for all adic spaces.
This can build upon the first sections of the present article but needs more input in the study of proper morphisms.

\section{Setup and notation}

In this article we will work with the tame and strongly étale site of pseudoadic spaces, mainly discretely ringed ones.
To put everyone on the same page, we start by reviewing these concepts (assuming familiarity with adic spaces).

Pseudoadic spaces have been introduced by Huber in order to study closed subsets of adic spaces, which do not necessarily carry the structure of an adic space.
Following \cite{Hu96}, \S~1.10, a \emph{pseudoadic space} is a pair $X = (\underline{X},\lvert X \rvert)$, where~$\underline{X}$ is an adic space and $\lvert X \rvert \subseteq \underline{X}$ is a convex proconstructible subset.
A \emph{morphism of pseudoadic spaces} from~$Y$ to~$X$ is a morphism of adic spaces $f: \underline{Y} \to \underline{X}$ such that $f(\lvert Y \rvert) \subseteq \lvert X \rvert$.
The morphism~$f$ is \emph{étale} if it is étale as a morphism of adic spaces and $\lvert Y \rvert$ is open in $f^{-1}(\lvert X \rvert)$.
It is \emph{strongly étale} if it is étale and for all points $y \in \lvert Y \rvert$ mapping to $x \in \lvert X \rvert$ the residue field extension $k(y)/k(x)$ is unramified.
Similarly, $f$ is \emph{tame} if it is étale and the above residue field extensions are tamely ramified.
Defining coverings to be surjective families, we get the étale, strongly étale, and tame sites~$X_\et$, $X_\set$, and~$X_t$ of a pseudoadic space~$X$.

A point~$x$ of a (pseudo)adic space~$X$ has a residue field~$k(x)$, which is a valued field with valuation ring~$k(x)^+$.
Now $k(x)^+$ again has a residue field.
In order to avoid confusion with $k(x)$ we will call it the \emph{specialization field} of~$x$ and write $k(x)^\succ$.
Because of the existence of both the residue field \emph{and} the specialization field of a point, there are two different notions of characteristic of an adic space~$X$.
The \emph{(residue) characteristic} $\ch(X)$ is the set of all primes occuring as the characteristic of a residue field $k(x)$ for some point $x \in X$.
The \emph{specialization characteristic} $\ch^+(X)$ is the set of all primes occuring as the characteristic of a specialization field $k(x)^\succ$ for some $x \in X$.

An adic space~$X$ is \emph{discretely ringed} if locally it is isomorphic to $\Spa(A,A^+)$ for a Huber pair $(A,A^+)$ whose topology is discrete.
A big class of examples is provided by spaces of the form $\Spa(X,S)$ for a morphism of schemes $X \to S$.
Its points are triples $(x,R,\phi)$ consisting of a point $x \in X$, a valuation ring~$R$ of its residue field~$k(x)$ and a morphism $\phi \colon \Spec R \to S$ fitting into the commutative diagram
\[
 \begin{tikzcd}
  \Spec k(x)	\ar[r]			\ar[d]	& X	\ar[d]	\\
  \Spec R		\ar[r,"\phi"]			& S.
 \end{tikzcd}
\]
In order to improve notation, we will denote the points of $\Spa(X,S)$ just by a single letter, e.g.~$x$.
Then the corresponding triple will be written as $(\supp x,k(x)^+,c_x)$.
This is in line with the notation for affinoid adic spaces $\Spa(A,A^+)$ where a valuation~$x$ is determined by its support $\supp x$ and the valuation ring $k(x)^+$ on $k(x) = k(\supp x)$.
In the affine case the map~$c_x$ is uniquely determined (or more generally if~$S$ is separated).
The existence of~$c_x$ is equivalent to the condition that the valuation~$x$ be less or equal to~$1$ on~$A^+$.

\section{Riemann-Zariski points and morphisms} \label{section_RZ_points}

In \cite{Tem11} Temkin studies adic spaces of the form $\Spa(X,S)$ for a morphism of schemes $X \to S$ and compares $\Spa(X,S)$ to the limit over all $X$-modifications of~$S$.
It turns out that this limit -- he calls it  $\RZ_X(S)$, the relative Riemann Zariski space -- is homeomorphic to a subspace of $\Spa(X,S)$.
It consists of all points that do not admit nontrivial horizontal specializations.
This description can easily be extended to more general adic spaces and this is what we will do in this section.

\subsection{Specialization and generalization of continuous valuations}

We start by reviewing some facts about specializations on adic spaces (to be found in \cite{Hu93}, \S~2), here and there adding some results that were not treated in the literature.
Remember that a specialization $x \rightsquigarrow y$ on an adic space factors uniquely as $x \rightsquigarrow z \rightsquigarrow y$, where $x \rightsquigarrow z$ is a vertical and $z \rightsquigarrow y$ is a horizontal specialization.
If $X = \Spa(A,A^+)$ is affinoid, vertical specializations of~$x$ can be characterized as those that have the same support in $\Spec A$, whereas horizontal specializations are those that have the same center in $\Spec A^+$.
We can also factor $x \rightsquigarrow y$ as $x \rightsquigarrow z_1 \rightsquigarrow z_2 \rightsquigarrow y$, where $x \to \rightsquigarrow z_1$ is a horizontal specialization and $z_2 \rightsquigarrow y$ is a vertical specialization.
If the valuation~$z_1$ is nontrivial, we necessarily have $z_1 = z_2$.
If~$z_1$ is trivial, then also~$z_2$ is trivial and we call $z_1 \rightsquigarrow z_2$ a schematic specialization.

\begin{lemma} \label{specialization_continuous}
 Let~$A$ be a Huber ring and $v:A \to \Gamma$ a continuous valuation.
 Then any specialization of $v$ as an element of $\Spv(A)$ is again continuous.
\end{lemma}

\begin{proof}
 From the proof of Theorem~3.1 in \cite{Hu93} we know that a valuation~$v$ is continuous if and only if for every $a \in A^{\circ\circ}$ its value $v(a)$ is cofinal for~$\Gamma_v$, i.e. for every $\gamma \in \Gamma_v$ there is $n \in \N$ such that $v(a^n) < \gamma$.
 As every specialization can be split into a horizontal and a vertical one, we can deal with horizontal and vertical specializations separately.
 If $v'$ is a horizontal specialization of~$v$, there is a convex subgroup $H$ of~$\Gamma_v$ containing the characteristic subgroup $c\Gamma_v$ such that $\Gamma_{v'} = H$ and
 \[
  v'(a) = \begin{cases}
           v(a)	& \text{if } v(a) \in H	\\
           0	& \text{else}.
          \end{cases}
 \]
 Using the above characterisation of continuous valuations, it is clear that~$v'$ is again continuous.
 If $v'$ is a vertical specialization of~$v$, there is a convex subgroup $H$ of~$\Gamma_{v'}$ such that $\Gamma_v = \Gamma_{v'}/H$ and
 \[
  v(a) = \begin{cases}
          v'(a)H	& \text{if } v'(a) \ne 0	\\
          0			& \text{if } v'(a) = 0.
         \end{cases}
 \]
Remember that the ordering of $\Gamma_v = \Gamma_{v'}/H$ is defined by saying that $\gamma < \delta$ if and only if $\gamma' < \delta'$ for all lifts~$\gamma'$ and~$\delta'$ of $\gamma$ and~$\delta$ to $\Gamma_{v'}$.
If $\gamma' \in \Gamma_{v'}$ with image $\gamma \in \Gamma_v$ and $a \in A$ is topologically nilpotent, then $v(a^n) < \gamma$ for some $n \in \N$.
This implies $v'(a^n) < \gamma'$, hence $v'$ is continuous.
\end{proof}

\begin{lemma} \label{generalization_continuous}
 Let~$A$ be a Huber ring and $v:A \to \Gamma$ a continuous valuation.
 A vertical generalization~$w$ of~$v$ in $\Spv(A)$ is continuous if and only if $w(A^{\circ\circ}) < 1$.
 In particular if~$v$ is non-analytic, any vertical generalization of~$v$ is continuous.
 If~$v$ is analytic, a vertical generalization is continuous if and only if it is not a trivial valuation.
\end{lemma}

\begin{proof}
 Remember the characterization of continuous valuations as those valuations~$v$ such that $v(a)$ is cofinal for $\Gamma_v$ for any $a \in A^{\circ\circ}$ (see \cite{Hu93}, Theorem~3.1 and its proof).
 From this it is clear that if $w$ is continuous, then $w(A^{\circ\circ}) < 1$.
 Assume now that $w(A^{\circ\circ}) < 1$.
 Since~$w$ is a vertical generalization of~$v$, there is a convex subgroup $H$ of~$\Gamma_v$ such that $\Gamma_w = \Gamma_v/H$ and
 \[
  w(a) = \begin{cases}
          v(a)H		& \text{if } v(a) \ne 0,	\\
          0			& \text{if } v(a) = 0.
         \end{cases}
 \]
 For $a \in A^{\circ\circ}$ and $\bar{\gamma} \in \Gamma_w$ we have to find $n \in \N$ with $w(a)^n < \bar{\gamma}$.
 Let $\gamma \in \Gamma_v$ be a preimage of~$\bar{\gamma}$.
 As~$v$ is continuous, there is $n \in \N$ such that $v(a^{n-1}) < \gamma$.
 Moreover, the condition $w(a) < 1$ translates to $v(a) < h$ for every $h \in H$.
 Combining the two inequalities gives $v(a^n) < h\gamma$ for all $h \in H$, or in other words $w(a^n) < \bar{\gamma}$.
 
 If~$v$ is non-analytic, $v(A^{\circ\circ}) = 0$.
 Hence $w(A^{\circ\circ}) = 0 < 1$ for any vertical generalization~$w$ of~$v$.
 Thus, any vertical generalization is continuous.
 
 Suppose now that~$v$ is analytic and let~$w$ be a vertical generalization of~$v$.
 We first notice that if~$w$ is a trivial valuation, it is not continuous.
 Indeed, as~$v$ is analytic, there is $a \in A^{\circ\circ}$ with $v(a) > 0$.
 But then also $w(a) > 0$.
 This implies $w(a) = 1$ as~$w$ is a trivial valuation.
 
 If~$w$ is non-trivial, the convex subgroup~$H$ of~$\Gamma_v$ defining~$w$ is a proper subgroup of~$\Gamma_v$: $H \subsetneq \Gamma_v$.
 In order to show that~$w$ is continuous, we need to convince ourselves that
 \[
  v(A^{\circ\circ}) < H.
 \]
 Let $a \in A^{\circ\circ}$ and $h \in H$.
 There is $\gamma \in \Gamma_v$ such that $\gamma < H$.
 Using that $v(a)$ is cofinal for~$\Gamma_v$, we find $n \in \N$ such that $v(a)^n < \gamma$.
 Since $\gamma < h^n$, we obtain $v(a) < h$.
\end{proof}

\subsection{Riemann-Zariski points} \label{subsection_RZ_points}

\begin{definition}
 A \emph{Riemann-Zariski (RZ) point} of a pseudoadic space~$X$ is a point which does not admit nontrivial horizontal specializations.
 \end{definition}

Any closed point is Riemann-Zariski.
Any point coresponding to a trivial valuation is a Riemann-Zariski point.
We call these \emph{trivial Riemann-Zariski points}.
Also every point of an analytic adic space is Riemann-Zariski.

\begin{lemma} \label{preimage_RZ}
 Let $f: U \to X$ be a locally quasi-finite morphism of pseudoadic spaces and $x \in U$ a point.
 If $f(x)$ is a Riemann-Zariski point, then $x$ as well.
\end{lemma}

\begin{proof}
 If $x'$ is a horizontal specialization of~$x$, then $f(x')$ is a horizontal specialization of~$f(x)$, hence equal to~$f(x)$.
 Since $f$ is locally quasi-finite, there are no specialization relations between the points of a fiber.
 This implies $x = x'$.
\end{proof}

\begin{lemma} \label{RZ_qc}
 Let~$X$ be a quasi-compact pseudoadic space.
 Then the subset of Riemann-Zariski points is quasi-compact.
\end{lemma}

\begin{proof}
 Let~$\underline{X}$ be the underlying adic space of~$X$ and~$|X|$ the subset of~$\underline{X}$ such that $X = (\underline{X},|X|)$.
 The adic space~$\underline{X}$ is locally spectral by \cite{Hu93}, Theorem~3.5~(i).
 Moreover, by definition of a pseudoadic space, $|X|$ is locally pro-constructible in~$\underline{X}$, hence locally spectral, as well.
 Being quasi-compact, $|X|$ is even spectral.
 The decisive property is that in a spectral space every point specializes to a closed point.
 Moreover, every closed point of~$|X|$ is Riemann-Zariski.
 Hence, every point of~$|X|$ specializes to a Riemann-Zariski point in~$|X|$.
 
 Denote by~$S$ the subset of Riemann-Zariski points of~$|X|$.
 Let
 \[
  S = \bigcup_{i \in I} S_i
 \]
 be an open cover.
 We need to show that it has a finite subcover.
 For every $i$ choose an open subset $U_i$ of $|X|$ such that $U_i \cap S = S_i$.
 Then $|X| = \bigcup_i U_i$ because a point in the complement of the union of the $U_i$ would specialize to some Riemann-Zariski point but every Riemann-Zariski point is contained in some~$U_i$.
 Since~$|X|$ is quasi-compact, finitely many~$U_i$'s cover it.
 But then the corresponding~$S_i$ cover~$S$.
\end{proof}

\begin{lemma} \label{separated_horizontal}
 Let~$X$ be separated over an affinoid adic space~$S$ and $x \in X$ a point.
 Then the set of horizontal specializations of~$x$ is totally ordered by specialization.
\end{lemma}

\begin{proof}
 First note that if~$X$ is affinoid, the horizontal specializations of a point $x \in X$ are in one-to-one correspondence with the convex subgroups of the value group $\Gamma_x$ containing the characteristic subgroup $c\Gamma_x$.
 In particular, they are totally ordered by specialization.
 
 In the general case, let $x_1$ and $x_2$ be horizontal specializations of~$x$ in~$X$.
 They correspond to convex subgroups~$\Delta_1$ and~$\Delta_2$ of the value group~$\Gamma_x$ of~$x$.
 Since the convex subgroups of~$\Gamma_x$ are totally ordered by inclusion, we may assume that $\Delta_1 \subseteq \Delta_2$.
 We claim that $x_2$ specializes to~$x_1$.
 
 The images $s_1$ and~$s_2$ of $x_1$ and~$x_2$ in~$S$ are horizontal specializations of the image~$s$ of~$x$.
 Since~$S$ is affinoid and $\Delta_1 \subseteq \Delta_2$, $s_2$ specializes to~$s_1$.
 Choose affinoid neighborhoods $U_1$, $U_2$, and~$V$ of $x_1$, $x_2$ and~$s_1$, respectively.
 Note that $s_2 \in V$ and~$x$ is contained in both~$U_1$ and~$U_2$.
 We denote by $c_1\Gamma_x$ and~$c_2\Gamma_x$ the characteristic subgroups of~$\Gamma_x$ with respect to~$U_1$ and~$U_2$, respectively.
 They are both contained in~$\Delta_2$.
 
 We consider the diagonal
 \[
  \Delta: X \to X \times_S X.
 \]
 The open subspace $U_1 \times_V U_2$ of $X \times_S X$ is affinoid and contains $\Delta(x)$.
 The value group of $\Delta(x)$ equals~$\Gamma_x$ and its characteristic subgroup $c_{12}\Gamma_x$ with respect to $U_1 \times_V U_2$ is the subgroup generated by $c_1\Gamma_x$ and~$c_2\Gamma_x$ (in fact one is contained in the other so it is just the bigger one of the two).
 Hence $c_{12}\Gamma_x \subseteq \Delta_2$ and there is a corresponding horizontal specialization $y_2 \in U_1 \times_V U_2$ of $\Delta(x)$.
 Its image in $U_2$ is the unique horizontal specialization of~$x$ corresponding to~$\Delta_2$, i.e. equal to~$x_2$.
 Since~$X$ is separated over~$S$, $\Delta$ is a closed immersion and thus~$y_2$ lies on the diagonal.
 Therefore the image of~$y_2$ in~$U_1$ also equals~$x_2$.
 In other words, $x_2$ is contained in~$U_1$.
 But as the horizontal specializations of~$x$ in the affinoid space~$U_1$ are totally ordered, this implies that~$x_1$ is a specialization of~$x_2$.
\end{proof}

define characteristic subgroup $c\Gamma_v$ for separated spaces. Relation to horizontal specializations.

\begin{lemma} \label{qc_horizontal}
 Let~$X$ be a quasi-compact pseudoadic space and $x \in X$ a point.
 Then there is a Riemann-Zariski point which is a horizontal specialization of~$x$.
\end{lemma}

\begin{proof}
 If~$X$ is affinoid, the horizontal specializations of~$x$ correspond to the convex subgroups of~$\Gamma_x$ containing the characteristic subgroup $c\Gamma_x$.
 The set of horizontal specializations has a maximal element~$y$ corresponding to the convex subgroup $c\Gamma_x$.
 This point $y$ is a Riemann-Zariski point and it is uniquely determined.
 
 In the general case, the set of horizontal specializations is partially ordered by specialization.
 Let
 \[
  X = \bigcup_i U_i
 \]
 be a covering of~$X$ by finitely many affinoid opens $U_i$.
 In every~$U_i$ containing~$x$ there is a unique Riemann-Zariski point~$x_i$ (with respect to~$U_i$, not necessarily with respect to~$X$) that is a horizontal specialization of~$x$.
 It corresponds to the characteristic subgroup~$c_i\Gamma_x$ of~$\Gamma_x$ computed with respect to~$U_i$.
 The finitely many convex subgroups~$c_i\Gamma_x$ of~$\Gamma_x$ are totally ordered and we can pick a minimal element~$c_j\Gamma_x$.
 Since the~$U_i$ cover~$X$, $x_j$ is a Riemann-Zariski point not only in~$U_j$ but also in~$X$.
\end{proof}

Combining \cref{separated_horizontal} with \cref{qc_horizontal}, we see that every point $x$ in a quasi-compact adic space that is separated over an affinoid one there is a unique Riemann-Zariski point which is a horizontal specialization of~$x$.
We denote it by~$x_{\RZ}$.

\subsection{Riemann-Zariski morphisms} \label{subsection_RZ_morphisms}

\begin{definition}
 A morphism of pseudoadic spaces is called Riemann-Zariski if it maps Riemann-Zariski points to Riemann-Zariski points.
\end{definition}

For instance, every closed immersion is Riemann Zariski and any morphism to an analytic adic space is Riemann Zariski.

\begin{lemma} \label{composition_RZ}
 The composition of two Riemann-Zariski morphisms is Riemann-Zariski.
\end{lemma}

\begin{proof}
 This is clear from the definition.
\end{proof}

\begin{lemma} \label{gf_RZ_then_f_RZ}
 Let $f: X \to Y$ and $g: Y \to Z$ be morphisms of pseudoadic spaces and suppose that~$g$ is locally quasi-finite.
 If $gf$ is Riemann-Zariski, then~$f$ as well.
\end{lemma}

\begin{proof}
 Let $x \in X$ be a Riemann-Zariski point with images $y \in Y$ and $z \in Z$.
 Then $g(y) = z$ and~$z$ is a Riemann-Zariski point.
 Since~$g$ is locally quasi-finite, \cref{preimage_RZ} implies that~$y$ is a Riemann-Zariski point.
\end{proof}

\begin{lemma} \label{RZ_stable_bc}
 Consider a Cartesian square
 \begin{equation} \label{diagram_bc_RZ}
  \begin{tikzcd}
   Y'	\ar[r,"\pi_Y"]	\ar[d,"f'"']	& Y	\ar[d,"f"]	\\
   X'	\ar[r,"\pi_X"']					& X
  \end{tikzcd}
 \end{equation}
 of quasi-compact quasi-separated pseudoadic spaces such that $f$ is Riemann-Zariski and $\pi_X$ locally quasi-finite.
 Then $f'$ is a Riemann-Zariski morphism.
\end{lemma}

\begin{proof}
 Let $y' \in Y'$ be a Riemann-Zariski point and assume that $x'= f'(y')$ is not Riemann-Zariski.
 Then $y = \pi_Y(y')$ isn't a Riemann-Zariski point either.
 Otherwise $x=f(y)$ would be a Riemann-Zariski point due to~$f$ being a Riemann-Zariski morphism.
 But by the commutativity of diagram~(\ref{diagram_bc_RZ}) we have $x = \pi_X(x')$.
 Since~$\pi_X$ is locally quasi-finite, this would imply by \cref{preimage_RZ} that $x'$ is a Riemann-Zariski point.
 
 Then there are convex subgroups $H \subset \Gamma_y$ and $H' \subset \Gamma_{x'}$ such that $y_{\RZ} = y|_H$ and $x'_{\RZ} = x'|_{H'}$.
 Since neither of the two points is Riemann-Zariski, the subgroups are proper subgroups.
 The value groups $\Gamma_y$ and~$\Gamma_{x'}$ are naturally embedded into $\Gamma_{y'}$.
 The convex subgroup $G \subset \Gamma_{y'}$ generated by the images of~$H$ and~$H'$ is again a proper subgroup.
 We claim that $c\Gamma_{y'} \subseteq G$, i.e., the horizontal specialization $y'|_G$ exists.
 
 Let~$\p$ be the prime ideal of the valuation ring $k(y')^+$ corresponding to~$G$.
 We equip $k(y')^+_{\p}$ with the topology induced from $k(y')$.
 Let $S \subseteq \Spa(k(y')^+_{\p},k(y')^+)$ be the closure of the point~$y'$.
 The points in~$S$ correspond to the quotients of $k(y')^+$ by prime ideals contained in~$\p$.
 We also write~$S$ for the pseudoadic space
 \[
  (\Spa(k(y')^+_{\p},k(y')^+),S).
 \]
 By construction we obtain a diagram (of solid arrows)
 \[
  \begin{tikzcd}
   S	\ar[rrd]	\ar[ddr]	\ar[rd,dashed]		\\
												& Y'	\ar[r,"\pi_Y"]	\ar[d,"f'"']	& Y	\ar[d,"f"]	\\
												& X'	\ar[r,"\pi_X"']					& X
  \end{tikzcd}
 \]
 and thus the dotted arrow exists.
 It maps the generic point of~$S$ to~$y'$ and the closed point to some nontrivial horizontal specialization of~$y$.
 But this is not possible as~$y'$ is a Riemann-Zariski point.
\end{proof}

\subsection{The maximal Riemann-Zariski open}

For an \'etale morphism $U \to X$ of separated and quasi-compact adic spaces we define the following subset of~$U$:
\[
 U_{\RZ} = \{ u \in U \mid f(u_{\RZ})~\text{is a Riemann-Zariski point}\}
\]
This construction is functorial in the following sense.
If $V \to X$ is another \'etale morphism as above and $V \to U$ is a morphism over~$X$, we obtain a factorization
\[
 V_{\RZ} \to U_{\RZ}.
\]
We would expect~$U_{\RZ}$ to be an open subset of~$U$.
However, this is not always the case as the following example shows.

\begin{example} \label{U_RZ_not_open}
 Let~$k$ be a field and set $A = k[S,T]$.
 We consider the discretely ringed affinoid adic space $\Spa(A,A)$.
 Let~$X$ be the union of the two affinoid open subspaces
 \[
  X_1 = \Spa(A_S,A) = \{x \in \Spa(A,A) \mid |S(x)| \ne 0\}
 \]
 and
 \[
  X_2 = \Spa(A_T,A_T) = \{x \in \Spa(A,A) \mid 1 \le |T(X)|\}.
 \]
 Let $x \in X$ be the point corresponding to the valuation of the function field of~$A$ associated to the prime divisor $V(S)$.
 It is contained in~$X_1$ and~$X_2$.
 Let $x_1 \in X_1$ be the vertical specialization of~$x$ obtained by composing the valuation~$v_x$ with the valuation of the function field of~$V(S)$ corresponding to the prime divisor $V(S,T)$ of $V(S)$.
 It is contained in~$X_1$ but not in~$X_2$.
 Let $x_2 \in X_2$ be the horizontal specialization of~$x$ corresponding to the trivial valuation of the function field of $V(S)$.
 We now consider the open embedding $X_1 \to X$.
 Obviously, $x_1 \in X_{1,\RZ}$ but $x \notin X_{1,\RZ}$ because $x$ is a Riemann-Zariski point of~$X_1$ but not a Riemann-Zariski point of~$X$ as it has the horizontal specialization~$x_2$ in~$X$.
 This shows that~$X_{1,\RZ}$ cannot be open in~$X_1$ since it contains~$x_1$ but not its generalization~$x$.
\end{example}

The decisive property the adic space in \cref{U_RZ_not_open} lacks is the following:

\begin{definition}
 An adic space~$X$ is square complete if for any three points $x$, $x_1$, and~$x_2$ in~$X$ such that~$x_1$ is a horizontal and~$x_2$ a vertical specialization of~$x$, there exists a unique point $x' \in X$ which is a vertical specialization of~$x_1$ and a horizontal specialization of~$x_2$:
 \[
  \begin{tikzcd}
   x_2	\ar[r,rightsquigarrow]							& x'			\\
   x	\ar[r,rightsquigarrow]	\ar[u,rightsquigarrow]	& x_1.	\ar[u,rightsquigarrow]
  \end{tikzcd}
 \]

\end{definition}

All analytic adic spaces are square complete since there are no nontrivial horizontal specializations.
Also affinoid adic spaces are square complete by \cite{HK94}, Lemma~1.2.5~i).
In fact, by the same argument as for affinoid adic spaces, all discretely ringed adic spaces of the form $\Spa(X,S)$ for a separated morphism of schemes $X \to S$ are square complete.

We will see below that~$U_{\RZ}$ is a quasi-compact open subset of~$U$ if~$X$ is square complete.
In order to prove this, we need some preparation.

\begin{lemma} \label{specialization_RZ}
 Suppose in a separated, quasi-compact, square complete adic space~$X$ a point $x \in X$ specializes to a Riemann-Zariski point $x'$ (not necessarily horizontally).
 Let
 \[
  x \rightsquigarrow x_1 \rightsquigarrow x_2 \rightsquigarrow x'
 \]
 be a decomposition, where $x \rightsquigarrow x_1$ is a horizontal specialization, $x_1 \rightsquigarrow x_2$ is a schematic specialization (or $x_1 = x_2$), and $x_2 \rightsquigarrow x'$ is a vertical specialization.
 Then $x_1 = x_{\RZ}$.
\end{lemma}

\begin{proof}
 Since the horizontal specializations are totally ordered by \cref{separated_horizontal}, $x_{\RZ}$ is a horizontal specialization of~$x_1$.
 If~$x_1$ corresponds to a trivial valuation, it is a Riemann-Zariski point and thus $x_1 = x_{\RZ}$.
 If it is non-trivial, there is no schematic specialization, and $x_1 = x_2$.
 As~$X$ is square complete, there is a unique point $x''$ such that $x_{\RZ} \rightsquigarrow x''$ vertically and $x' \rightsquigarrow x''$ horizontally.
 Since~$x'$ is a Riemann-Zariski point, $x' = x''$.
 This also implies $x_1 = x_{\RZ}$.
\end{proof}

\begin{lemma} \label{cover_source}
 Let $f:U \to X$ be an \'etale morphism of separated, quasi-compact adic spaces.
 Let $U = \bigcup_i U_i$ be an open cover.
 Then
 \[
  U_{\RZ} = \bigcup_i U_{i,\RZ}.
 \]
\end{lemma}

\begin{proof}
 In order to show the inclusion $U_{i,\RZ} \subseteq U_{\RZ}$, take a point $x \in U_{i,\RZ}$.
 Then $x_{\RZ,U_i}$ (computed in~$U_i$) maps to a Riemann-Zariski point of~$X$.
 Now we view~$x_{\RZ,U_i}$ as a point of~$U$ and apply \cref{preimage_RZ} to conclude that it is a Riemann-Zariski point of~$U$, as well.
 This shows $x \in U_{\RZ}$.
 
 Now let $x$ be an element of~$U_{\RZ}$.
 Then $x_{\RZ,U}$ is contained in one of the~$U_i$'s.
 Also~$x$ is contained in this~$U_i$ and $x_{\RZ,U}$ is the Riemann-Zariski point of~$x$ computed in~$U_i$.
 By assumption~$x_{RZ,U}$ maps to a Riemann-Zariski point of~$X$.
 Hence, $x \in U_{i,\RZ}$.
\end{proof}

\begin{lemma} \label{bijection_horizontal}
 Let $f: U \to X$ be an \'etale morphism of separated, quasi-compact adic spaces and $u \in U$ a point.
 Then~$f$ induces a bijection of the horizontal specializations of~$u$ with the horizontal specializations of~$f(u)$ dominated by $f(u_{\RZ})$.
\end{lemma}

\begin{proof}
 Without loss of generality, we may replace~$X$ and~$U$ by open subspaces~$X'$ and~$U'$ containing $f(u)_{\RZ}$ and~$u_{\RZ}$, respectively, such that $f(U') \subseteq X'$.
 In this way we reduce to the case where~$U$ and~$X$ are affinoid.
 Since~$f$ is \'etale, the value groups of~$u$ and~$f(u)$ are the same.
 Moreover, the characteristic subgroup of~$u$ is the value group of~$u_{\RZ}$, which in turn coincides with the value group of $f(u_{\RZ})$.
 Hence, there is a one-to-one correspondence between the horizontal specializations of~$u$ in~$U$ and of~$f(u)$ in~$X$ dominated by $f(u_{\RZ})$.
\end{proof}

\begin{lemma} \label{RZ_point_subset}
 Let $f: U \to X$ be an \'etale morphism of separated, quasi-compact, and square complete adic spaces.
 Take an open subset $X'$ of~$X$ and set $U' = U \times_X X'$.
 Let $u' \in U'$ be a Riemann-Zariski point and let $u \in U$ be be a specialization of~$u'$.
 If $f(u)$ is a Riemann-Zariski point of~$X$, then $f(u')$ is a Riemann-Zariski point of~$X'$.
\end{lemma}

\begin{proof}
 In~$U$ we can decompose the specialization $u' \rightsquigarrow u$ as
 \[
  u' \rightsquigarrow u_1 \rightsquigarrow u_2 \rightsquigarrow u,
 \]
 where $u' \rightsquigarrow u_1$ is a horizontal specialization, $u_1 \rightsquigarrow u_2$ is a schematic specialization, and $u_2 \rightsquigarrow u$ is a vertical specialization (see \cite{HK94}, Proposition~1.2.4).
 By \cref{preimage_RZ}, $u$ is a Riemann-Zariski point of~$U$.
 Therefore, \cref{specialization_RZ} implies that $u'_{\RZ,U} = u_1$
 The above chain of specializations maps to a chain of specializations
 \[
  f(u') \rightsquigarrow f(u_1) \rightsquigarrow f(u_2) \rightsquigarrow f(u).
 \]
 of the same kind and by the same reason as before we have $f(u_1) = f(u')_{\RZ,X}$.
 The Riemann-Zariski point of~$X'$ associated with~$f(u')$ is a horizontal specialization of~$f(u')$.
 By \cref{bijection_horizontal} it corresponds to a horizontal specialization~$u''$ of~$u'$, i.e. $f(u'') = f(u')_{\RZ,X'}$
 As $f^{-1}(X') = U'$, we know that $u'' \in U'$.
 But~$u'$ is a Riemann-Zariski point of~$U'$ and thus $u'' = u'$ and $f(u')_{\RZ,X'} = f(u')$.
\end{proof}

\begin{lemma} \label{U_RZ_local_base}
 Let $f: U \to X$ be an \'etale morphism of separated, quasi-compact adic spaces.
 Let $X = \bigcup_i X_i$ be an open cover and set $U_i = X_i \times_X U$.
 Then an open subset $V \subseteq U$ is contained in $U_{\RZ}$ if and only if $V \cap U_i$ is contained in $U_{i,\RZ}$ (defined with respect to $f_i : U_i \to X_i$) for all~$i$.
\end{lemma}

\begin{proof}
 Suppose $V \subseteq U_{\RZ}$ and let us show $V \cap U_i \subseteq U_{i,\RZ}$.
 Take $u \in V \cap U_i$.
 By assumption $f(u_{\RZ,U})$ is a Riemann-Zariski point of~$X$.
 We have a chain of horizontal specializations
 \[
  f(u) \rightsquigarrow f(u_{\RZ,U_i}) \rightsquigarrow f(u_{\RZ,U})
 \]
 By \cref{bijection_horizontal}, any horizontal specialization of~$f(u_{\RZ,U_i})$ in~$X_i$ corresponds to a horizontal specialization~$u'$ of $u_{\RZ,U_i}$ in~$U$.
 But since $f^{-1}(X_i) = U_i$, we must have $u' \in U_i$ and thus $u' = u_{\RZ,U_i}$.
 This shows that $f(u_{\RZ,U_i})$ is a Riemann-Zariski point of~$X_i$.
 In other words, $u \in U_{i,\RZ}$.
 
 Now assume that $V \cap U_i \subseteq U_{i,\RZ}$ for all~$i$.
 For $u \in V$ we need to show that $f(u_{\RZ,U})$ is a Riemann-Zariski point.
 There is some~$i$ such that $f(u)_{\RZ,X} \in X_i$.
 Then also $u_{\RZ,U} \in U_i$, hence $u_{\RZ,U} = u_{\RZ,U_i}$.
 We obtain a chain of horizontal specializations in~$X_i$:
 \[
  f(u) \rightsquigarrow f(u_{\RZ,U_i}) = f(u_{\RZ,U}) \rightsquigarrow f(u)_{\RZ,X_i} = f(u)_{\RZ,X}.
 \]
 By assumption $f(u_{\RZ,U_i})$ is a Riemann-Zariski point of~$X_i$.
 This implies $f(u_{\RZ,U}) = f(u)_{\RZ,X}$.
\end{proof}

\begin{proposition} \label{U_RZ_open}
 Let $f: U \to X$ be an \'etale morphism of separated and quasi-compact adic spaces and suppose that~$X$ is square complete.
 Then $U_{\RZ}$ is a quasi-compact open subset of~$U$.
\end{proposition}

\begin{proof}
 We first show that~$U_{\RZ}$ is open.
 Let us start with the case where~$U$ and~$X$ are affinoid and discretely ringed, i.e., $X = \Spa(A,A^+)$ and $U = \Spa(B,B^+)$ and $A$ and~$B$ carry the discrete topology.
 Since $u \in U_{\RZ}$ if and only if $u_{\RZ} \in U_{\RZ}$, it suffices to construct for each Riemann-Zariski point~$u$ of~$U$ an open neighborhood that is contained in~$U_{\RZ}$.
 Using \cite{HueSch20}, Lemma~12.10, we construct an open affinoid neighborhood $V' \subseteq U$ of~$u$ such that $V' \to X$ is a Riemann-Zariski morphism.
 In particular, $V \subset U_{\RZ}$.
 
 Let us now lift the restriction that the topology of~$A$ and~$B$ be discrete.
 We denote by $(A,A^+)_{\disc}$ and $(B,B^+)_{\disc}$ the same Huber pairs but equipped with the discrete topology.
 The homomorphism of Huber pairs $(A,A^+)_{\disc} \to (A,A^+)$ induces a continuous inclusion
 \[
  \iota_: X = \Spa(A,A^+) \hookrightarrow \Spa(A,A^+)_{\disc} =: X_{\disc}
 \]
 and similarly for~$U$.
 We obtain a commutative diagram
 \[
  \begin{tikzcd}
   U	\ar[r,hookrightarrow,"\iota_U"]	\ar[d,"f"']	& U_{\disc}	\ar[d,"f_{\disc}"]	\\
   X	\ar[r,hookrightarrow,"\iota_X"]				& X_{\disc}.
  \end{tikzcd}
 \]
 It follows from \cref{specialization_continuous} that a point $x \in X$ is Riemann-Zariski if and only if $\iota_X(x)$ is Riemann-Zariski in~$X_{\disc}$ and similarly for~$U$.
 Therefore,
 \[
  U_{\RZ} = \iota_U^{-1}(U_{\disc,\RZ}) 
 \]
 But $U_{\disc,\RZ}$ is open, hence also $U_{\RZ}$.
 
 If~$X$ is affinoid and~$U$ not necessarily, we cover~$U$ by finitely many affinoids~$U_i$.
 By what we have just seen, $U_{i,\RZ}$ is open in~$U_i$, hence also in~$U$.
 By \cref{cover_source} we know that $U_{\RZ} = \bigcup U_{i,\RZ}$.
 In particular, $U_{\RZ}$ is open.

 In the general case where neither~$U$ nor~$X$ needs to be affinoid, we cover~$X$ by finitely many affinoids~$X_i$ and set $U_i = U \times_X X_i$.
 For every $i$ we denote by $Z_i \subseteq U_i$ the complement of~$U_{i,\RZ}$.
 This is a locally closed subset of~$U$ by what we have seen above.
 Hence, its closure~$\bar{Z}_i$ in~$U$ is given by all specializations of points in~$Z_i$.
 We claim that
 \[
  U' := U \setminus \bigcup_i \bar{Z}_i = U_{\RZ}.
 \]
 In order to show that $U' \subseteq U_{\RZ}$ we use \cref{U_RZ_local_base}.
 By definition
 \[
  U' \cap U_i = U_i \setminus \bigcup_j \big(U_i \cap \bar{Z}_j) \subseteq U_i \setminus Z_i = U_{i,\RZ}.
 \]
 for all~$i$.
 This implies $U' \subseteq U_{\RZ}$.
 
 Suppose there is a point $u \in U_{\RZ} \setminus U'$.
 As~$U'$ is open and~$u_{\RZ} \in U_{\RZ}$, we may assume that~$u$, and hence also $f(u)$, is a Riemann-Zariski point.
 By definition of~$U'$, $u$ is a specialization of some point $z \in Z_i$ for some~$i$.
 Let
 \[
  z \rightsquigarrow z_1 \rightsquigarrow z_2 \rightsquigarrow u,
 \]
 be the decomposition of the specialization $z \rightsquigarrow u$ into a horizontal specialization followed by a schematic and a vertical specialization.
 Then
 \[
  f(z) \rightsquigarrow f(z_1) \rightsquigarrow f(z_2) \rightsquigarrow f(u)
 \]
 gives the analogous decomposition for the specialization $f(z) \rightsquigarrow f(u)$.
 Since $f(u)$ is a Riemann-Zariski point and~$X$ is square complete, \cref{specialization_RZ} implies that $f(z_1) = f(z)_{\RZ,X}$.
 Then by \cref{preimage_RZ}, $z_1 = z_{\RZ,U}$.
 Using \cref{RZ_point_subset}, we conclude that $f(z_{\RZ,U_i})$ is a Riemann-Zariski point of~$X_i$.
 In other words, $z \in U_{i,\RZ}$, a contradiction.

 We now know that $U_{\RZ}$ is open and to finish the proof we have to convince ourselves that it is also quasi-compact.
 Let~$S$ be the set of Riemann-Zariski points of~$X$.
 It is quasi-compact by \cref{RZ_qc}.
 Moreover, $f$ being quasi-compact, the preimage $f^{-1}(S)$ is also quasi-compact.
 Inspecting the definition of~$U_{\RZ}$ and using \cref{preimage_RZ}, we see that $f^{-1}(S)$ is in fact precisely the set of Riemann-Zariski points of~$U_{\RZ}$.
 If
 \[
  U_{\RZ} = \bigcup_i U_i
 \]
 is an open cover of~$U_{\RZ}$, we know that finitely many~$U_i$ cover the Riemann-Zariski points of~$U_{\RZ}$.
 But any closed subset of~$U_{\RZ}$ contains a Riemann-Zariski point by \cref{qc_horizontal}, so these finitely many~$U_i$ cover all of~$U_{\RZ}$.
\end{proof}

In the situation of \cref{U_RZ_open} we know that $U_{\RZ}$ is open.
It is then clear from the construction that it is the maximal open subset~$V$ of~$U$ such that $V \to X$ is a Riemann-Zariski morphisms.
We call~$U_{\RZ}$ the \emph{maximal Riemann-Zariski open}.

\begin{lemma} \label{functorial_U_RZ}
 Let $X$ be a quasi-compact, separated, square complete pseudoadic space and $f:U \to V$ a morphism of \'etale $X$-spaces.
 Then~$f$ induces a morphism $f_{\RZ} : U_{\RZ} \to V_{\RZ}$.
\end{lemma}

\begin{proof}
 Let $u \in U_{\RZ}$, i.e. the image of~$u_{\RZ}$ in~$X$ is a Rieman Zariski point.
 Then $f(u_{\RZ})$ maps to a Riemann-Zariski point in~$X$, hence is itself a Riemann-Zariski point by \cref{preimage_RZ}.
 In particular $f(u)_{\RZ} = f(u_{\RZ})$ and this point maps to a Riemann-Zariski point in~$X$.
\end{proof}

\begin{lemma}
 Let $X$ be a quasi-compact, separated, square complete pseudoadic space and $f:U \to W$ and $g: V \to W$ morphisms of separated, quasi-compact \'etale $X$-spaces.
 Then there is a natural isomorphism
 \[
  (U \times_W V)_{\RZ} \cong U_{\RZ} \times_{W_{\RZ}} V_{\RZ}.
 \]
\end{lemma}

\begin{proof}
 By \cref{gf_RZ_then_f_RZ} the induced morphisms $U_{\RZ} \to W_{\RZ}$ and $V_{\RZ} \to W_{\RZ}$ are Riemann-Zariski.
 Hence the fiber product $U_{\RZ} \times_{W_{\RZ}} V_{\RZ}$ is Riemann-Zariski over~$W_{\RZ}$ by \cref{RZ_stable_bc}.
 Also the composition
 \[
  U_{\RZ} \times_{W_{\RZ}} V_{\RZ} \to W_{\RZ} \to X
 \]
 is Riemann-Zariski by \cref{composition_RZ}.
 Moreover, $U_{\RZ} \times_{W_{\RZ}} V_{\RZ}$ is an open subspace of $U \times_W V$.
 Since $(U \times_W V)_{\RZ}$ is the biggest open subspace of $U \times_W V$ which is Riemann-Zariski over~$X$, we obtain an inclusion
 \[
  U_{\RZ} \times_{W_{\RZ}} V_{\RZ} \subseteq (U \times_W V)_{\RZ}.
 \]
 In order to show the converse inclusion, we just need to notice that we have a natural morphism
 \[
  (U \times_W V)_{\RZ} \to U_{\RZ} \times_{W_{\RZ}} V_{\RZ}
 \]
 by \cref{functorial_U_RZ}, which is obviously inverse to the above inclusion.
\end{proof}

\section{Comparison with Riemann-Zariski spaces} \label{comparison_RZ}

Let $X \to S$ be a separated morphism of quasi-separated quasi-compact schemes.
In \cite{Tem11} Temkin establishes the following connection between the discretely ringed adic space $\Spa(X,S)$ and the Riemann-Zariski space $\RZ_X(S)$.
He constructs two continuous morphisms
\[
 \RZ_X(S) \overset{\iota}{\longrightarrow} \Spa(X,S) \overset{\pi}{\longrightarrow} \RZ_X(S)
\]
such that $\pi \circ \iota$ is the identity.
In particular, $\iota$ is a topological embedding.
Its image comprises precisely the Riemann-Zariski points of $\Spa(X,S)$.
Moreover, the composition $\iota \circ \pi$ maps a point $x \in \Spa(X,S)$ to its associated Riemann-Zariski point~$x_{\RZ}$.

We will generalize this construction in two directions.
Firstly, we consider more general adic spaces, not just discretely ringed ones.
Secondly, we lift the two maps~$\iota$ and~$\pi$ to morphisms of sites for the \'etale, the strongly étale, and the tame topology.

For an adic space $X$ we define the following site~$X_{\tau,\RZ}$ for $\tau \in \{\et,t,\set\}$:
Its underlying category is the full subcategory of~$X_{\tau}$ consisting of objects $U \to X$ that are Riemann-Zariski.
Covers are given by surjective families.
Note that fiber products exist in~$X_{\tau,\RZ}$ by \cref{RZ_stable_bc}.

The site $X_{\tau,\RZ}$ should be thought of as the $\tau$-site of the Riemann-Zariski points of~$X$.
We  always have a natural morphism of sites
\[
 \pi: X_{\tau} \to X_{\tau,\RZ}
\]
that is induced by the inclusion functor of $X_{\tau,\RZ}$ into $X_{\tau}$.
If~$X$ is square complete (e.g. affinoid, analytic or discretely ringed of the form $\Spa(Y,S)$), separated and quasi-compact, we have a morphism of sites
\begin{align*}
 \iota:	& X_{\tau,\RZ}	 \to X_{\tau},	\\
		& U_{\RZ}		\mapsfrom U,
\end{align*}
where~$U_{\RZ}$ is the open subspace of~$U$ defined in \cref{section_RZ_points}.
It is indeed an open quasi-compact subspace of~$U$ by \cref{U_RZ_open}.
Moreover, the construction is functorial by \cref{functorial_U_RZ}.

\begin{lemma} \label{iota_morphism_of_sites}
 The above functor~$\iota$ defines a morphism of sites.
\end{lemma}

\begin{proof}
 Let $(U_i \to U)_i$ be a cover in~$X_{\tau}$.
 We need to show that $(U_{i,\RZ} \to U_{\RZ})_i$ is a cover in~$X_{\tau,\RZ}$, i.e. that it is a surjective family.
 Since every point of~$U$ specializes to a Riemann-Zariski point and $U_{i,\RZ} \to U_{\RZ}$ is open for all~$i$, it suffices to show that the Riemann-Zariski points of~$U_{\RZ}$ are contained in the image.
 Let $u \in U_{\RZ}$ be a Riemann-Zariski point.
 By definition of~$U_{\RZ}$ it maps to a Riemann-Zariski point $x \in X$.
 As the $U_i$'s cover~$U$, $u$ has a preimage~$u_i$ in some $U_i$.
 It maps to the Riemann-Zariski point $x \in X$.
 By the definition of~$U_{i,\RZ}$, $u_i$ is thus contained in~$U_{i,\RZ}$.
\end{proof}

We want to show next that the morphisms of sites~$\pi$ and~$\iota$ are functorial in~$X$.
First we need to know that the Riemann-Zariski site is functorial:

\begin{lemma}
 A Riemann-Zariski morphism $f:Z \to X$ induces a morphism of sites
 \[
  Z_{\tau,\RZ} \longrightarrow X_{\tau,\RZ}
 \]
\end{lemma}

\begin{proof}
 Let $U \to X$ be a $\tau-\RZ$-morphism.
 Then
 \[
  U_Z := U \times_X Z \longrightarrow Z
 \]
 is a $\tau$-morphism by \cite{HueAd}, \S~3.
 It is also a Riemann-Zariski morphism by \cref{RZ_stable_bc}.
\end{proof}

It is clear that for a composition $Z \to Y \to X$ the composition $Z_{\tau,\RZ} \to Y_{\tau,\RZ} \to X_{\tau,\RZ}$ is equivalent to $Z_{\tau,\RZ} \to X_{\tau,\RZ}$ obtained from $Z \to X$.

\begin{lemma} \label{functoriality_iota_pi}
 For a Riemann-Zariski morphism $f:Z \to X$ the following diagram commutes.
 \[
  \begin{tikzcd}
   Z_\tau	\ar[r,"\pi_Z"]	\ar[d,"f_\tau"]	& Z_{\tau,\RZ}	\ar[r,"\iota_Z"]	\ar[d,"f_{\tau,\RZ}"]	& Z_\tau	\ar[d,"f_{\tau}"]	\\
   X_\tau	\ar[r,"\pi_X"]					& X_{\tau,\RZ}	\ar[r,"\iota_X"]							& X_\tau.
  \end{tikzcd}
 \]
\end{lemma}

\begin{proof}
 The left hand square commutes by construction of the morphisms of sites involved.
 The commutativity of the right hand square amounts to checking the identity
 \[
  (U \times_X Z)_{\RZ,Z} = U_{\RZ,X} \times_X Z
 \]
 as open subspaces of $U \times_X Z$ for any $\tau$-morphism $U \to X$.
 By compatibility with base change (\cref{RZ_stable_bc}) we get an inclusion
 \[
  U_\RZ \times_X Z \subseteq (U \times_X Z)_\RZ.
 \]
 Let us show the reverse inclusion.
 By \cref{RZ_stable_bc} the morphism
 \[
  f_U : U \times_X Z \to U
 \]
 is Riemann-Zariski.
 Therefore, for $y \in U \times_X Z$, we have $f_U(y_\RZ) = f_U(y)_\RZ$, which implies that~$f_U$ induces a morphism
 \[
  f_{U,\RZ} : (U \times_X Z)_\RZ \to U_\RZ.
 \]
 We get a diagram
 \[
  \begin{tikzcd}
   (U \times_X Z)_\RZ	\ar[ddr,bend right]	\ar[dr,dashed]	\ar[drr,bend left]	&									&\\
																				& U \times_X U_\RZ	\ar[r]	\ar[d]	& U_\RZ	\ar[d]	\\
																				& Z 				\ar[r]			& X
  \end{tikzcd}
 \]
 and the dashed arrow is the desired inclusion.
\end{proof}

\begin{lemma} \label{pi_*_exact}
 For every sheaf~$\mathcal{F}$ on~$X_{\tau}$ we have a natural isomorphism $R\pi_*\mathcal{F} \cong \iota^*\mathcal{F}$.
 In particular, $\pi_*$ is exact.
\end{lemma}

\begin{proof}
 Since~$i^*$ is exact, it suffices to show that $\pi_* = \iota^*$.
 For a sheaf~$\mathcal{F}$ on~$X_{\tau}$ the pullback $\iota^*\mathcal{F}$ is the sheafification of the presheaf on~$X_{\tau,\RZ}$ defined by
 \[
  U \mapsto \colim_{U \to V_{\RZ}} \mathcal{F}(V),
 \]
 where the colimit is taken over all morphisms $U \to V_{\RZ}$ in $X_{\tau,\RZ}$ for some object~$V$ of~$X_{\tau}$.
 This category has an initial object, namely the identity on~$U$.
 But the presheaf $U \mapsto \mathcal{F}(U)$ on~$X_{\tau,\RZ}$ coincides with the sheaf $\pi_*\mathcal{F}$.
\end{proof}

\begin{lemma} \label{bc_pi_RZ}
 For a Riemann Zariski morphism $f \colon Z \to X$ and a sheaf $\cF \in \Shv(X_\tau)$ the base change homomorphism
 \[
  f_{\tau,\RZ}^* \pi_{X,*} \cF \longrightarrow \pi_{Z,*} f_\tau^* \cF
 \]
 associated to the square
 \[
  \begin{tikzcd}
   Z_\tau	\ar[r,"\pi_\cZ"]	\ar[d,"f_\tau"']	& Z_{\tau,RZ}	\ar[d,"f_{\tau,RZ}"]	\\
   X_\tau	\ar[r,"\pi_*"']							& X_{\tau,RZ}
  \end{tikzcd}
 \]
 is an isomorphism.
\end{lemma}

\begin{proof}
 This follows from \cref{functoriality_iota_pi} and \cref{pi_*_exact}.
\end{proof}

\begin{proposition} \label{adjunction_identity}
 The adjunction map $\id \to \pi_*\pi^*$ is the identity.
\end{proposition}

\begin{proof}
 For a sheaf~$\mathcal{F}$ on~$X_{\tau,\RZ}$ and $U \in X_{\tau,\RZ}$ we have $\pi_*\pi^*\mathcal{F}(U) = \pi^*\mathcal{F}(U)$.
 We denote by $\pi^*_P\mathcal{F}$ the presheaf-pullback of~$\mathcal{F}$.
 For an object $U$ of~$X_{\tau}$ it is defined as
 \[
  \pi^*_P\mathcal{F}(U) = \colim_{U \to V} \mathcal{F}(V),
 \]
 where the colimit runs over all morphisms $U \to V$ in~$X_{\tau}$ with $V \in X_{\tau,\RZ}$.
 If~$U$ itself is an object of~$X_{\tau,\RZ}$, it is an initial object of the index category and thus $\pi^*_P\mathcal{F}(U) = \mathcal{F}(U)$.
 By the subsequent \cref{sheafification_chech}, it follows that for $U \in X_{\tau,\RZ}$
 \[
  \pi^*\mathcal{F}(U) = \mathcal{F}(U).
 \]
\end{proof}

\begin{lemma} \label{sheafification_chech}
 Let~$\mathcal{G}$ be a presheaf on~$X_{\tau}$ whose restriction~$\mathcal{F}$ to~$X_{\tau,\RZ}$ is a sheaf.
 Then for the sheafification~$\mathcal{G}^\#$ and $U \in X_{\tau,\RZ}$ we have
 \[
  \mathcal{G}^\#(U) = \mathcal{F}(U).
 \]
\end{lemma}

\begin{proof}
 The sheafification $\mathcal{G}^\#$ is given by applying twice the \v{C}ech global sections functor $\check{H}^0$ to $\mathcal{G}$.
 For $U \in X_{\tau}$ we have
 \[
  \check{H}^0(U,\mathcal{G}) = \colim_{(U_i \to U)} \check{H}^0(\{U_i \to U\},\mathcal{G}),
 \]
 where the colimit is taken over all covers of~$U$ in~$X_{\tau}$.
 If~$U$ is contained in~$X_{\tau,\RZ}$, the covers of~$U$ by objects of~$X_{\tau,\RZ}$ are cofinal among all covers by \cref{iota_morphism_of_sites}.
 Moreover, for~$U$ in~$X_{\tau,\RZ}$ we have $\mathcal{G}(U) = \mathcal{F}(U)$.
 For a cover $\{U_i \to U\}_i$ in $X_{\tau,\RZ}$ we thus have
 \[
  \check{H}^0(\{U_i \to U\},\mathcal{G}) = \check{H}^0(\{U_i \to U\},\mathcal{F}) = \mathcal{F}(U).
 \]
 Combining these two pieces of information we obtain for $U \in X_{\tau,\RZ}$:
 \[
  \check{H}^0(U,\mathcal{G}) = \mathcal{F}(U).
 \]
 In other words, $\check{H}^0(-,\mathcal{G})$ is again a presheaf whose restriction to~$X_{\tau,\RZ}$ is the sheaf~$\mathcal{F}$.
 Applying~$\check{H}^0$ a second time, thus yields by the same argument once more a presheaf with this property.
 Moreover, this presheaf is the sheafification~$\mathcal{G}^\#$.
\end{proof}

\begin{proposition} \label{adjunction_isomorphism}
 For a sheaf~$\mathcal{F}$ on~$X_{\tau}$ and $U \in X_{\tau,\RZ}$ the adjunction map $\epsilon(\mathcal{F}): \pi^*\pi_*\mathcal{F} \to \mathcal{F}$ induces isomorphisms
 \[
  H^i(U,\pi^*\pi_*\mathcal{F}) \overset{\sim}{\longrightarrow} H^i(U,\mathcal{F})
 \]
 for all $i \in \N$.
\end{proposition}

\begin{proof}
 We denote by $\epsilon: \pi^*\pi_* \to \id$ the counit and by $\eta: \id \to \pi_*\pi^*$ the unit of the adjunction $\pi^*\dashv \pi_*$.
 Then the composition
 \[
  \pi_* \overset{\eta}{\longrightarrow} \pi_*\pi^*\pi_* \overset{\epsilon}{\longrightarrow} \pi_*
 \]
 is the identity transformation.
 Moreover, the left hand transformation is an isomorphism by \cref{adjunction_identity}.
 Hence, the right hand transformation is an isomorphism as well.
 Consider the following commutative diagram of natural maps
 \[
  \begin{tikzcd}
   H^i(U,\pi^*\pi_*\mathcal{F})	\ar[r,"\text{edge}"]	\ar[d,"\epsilon"]	& H^i(U,\pi_*\pi^*\pi_*\mathcal{F})	\ar[d,"\epsilon"]\\
   H^i(U,\mathcal{F})			\ar[r,"\text{edge}"]						& H^i(U,\pi_*\mathcal{F})	.
  \end{tikzcd}
 \]
 The horizontal homomorphisms are isomorphisms by \cref{pi_*_exact} and the right vertical homomorphism is an isomorphism by what we have just seen.
 We conclude that also the left vertical homomorphism is an isomorphism.
\end{proof}

\cref{adjunction_isomorphism} allows us to restrict our attention to sheaves of the form $\pi^*\mathcal{F}$.

\section{Specialization} \label{sect_specialization}

Let $X \to S$ be a separated morphism of qcqs schemes.
For this general class of morphisms Temkin introduced in \cite{Tem11} the following concept of a modification:
An $X$-modification of~$S$ is a scheme~$Y$ fitting into a commutative diagram
\[
 \begin{tikzcd}
  X	\ar[r]	\ar[d]	& Y	\ar[dl]	\\
  S,
 \end{tikzcd}
\]
such that $X \to Y$ is dominant and $Y \to S$ is proper.
We set
\[
 \cX = \Spa(X,S).
\]
For every $X$-modification~$S_i$ of~$S$ we have a spectral map
\[
 \cX = \Spa(X,S) \to S_i
\]
that maps a valuation to its center.
These maps are compatible by the valuative criterion for properness and we obtain a map to the limit:
\[
 \Phi: \Spa(X,S) \to \RZ_X(S):= \lim_i S_i.
\]
Writing $\pi_i : \RZ_X(S) \to S_i$ for the projections, we have a natural sheaf
\[
 \cO_{\RZ_X(S)} = \colim_i \pi_i^{-1}\cO_{S_i}
\]
on~$\RZ_X(S)$.
For $x \in \RZ_X(S)$ we denote its image in~$S_i$ by~$s_i$.
If $S_j \to S_i$ is a transition map, we have by construction a map of local rings
\[
 \cO_{S_i,s_i} \to \cO_{S_j,s_j}
\]
Therefore,
\[
 \cO_{\RZ_X(S),x} = \colim_i \cO_{S_i,s_i}
\]
is local and $(\RZ_X(S),\cO_{\RZ_X(S)})$ is the limit of $(S_i,\cO_{S_i})$ in the category of locally ringed spaces.

\begin{proposition}
 The morphism
 \[
  \Phi: \Spa(X,S) \longrightarrow \RZ_X(S)
 \]
 restricts to a homeomorphism
 \[
  \Phi_\RZ : \cX_{\RZ} = \Spa(X,S)_{\RZ} \longrightarrow \RZ_X(S).
 \]
 For every point $x \in \Spa(X,S)_{\RZ}$ we have a natural isomorphism
 \[
  \cO^+_{\cX,x} \cong \colim_i \cO_{S_i,s_i},
 \]
 where the colimit runs over all $X$-modifications $S_i$ of~$S$ and~$s_i$ denotes the image of~$x$ in~$S_i$.
\end{proposition}

\begin{proof}
 These are the main results of \cite{Tem11}.
 See in particular Proposition~2.2.1 and Corollary~3.4.7 in loc. cit.
\end{proof}

We want to extend these results to an equivalence of the étale site of $\RZ_X(S)$ and the Riemann-Zariski strongly étale site of $\Spa(X,S)$.

Let $Y$ be an $X$-modification of~$S$.
We want construct a morphism of sites
\[
 \Psi_Y: \Spa(X,S)_{\set,\RZ} \longrightarrow Y_{\et}
\]
by mapping an object $Y' \to Y$ of~$Y_\et$ to
\[
 \Spa(Y'\times_Y X,Y') \to \Spa(X,S). 
\]

\begin{lemma} \label{psi_morphism_sites}
 The above assignment defines a morphism of sites $\Psi_Y: \Spa(X,S)_{\set,\RZ} \longrightarrow Y_{\et}$
\end{lemma}

\begin{proof}
 First, we need to convince ourselves that $\Spa(Y'\times_Y X,Y') \to \Spa(X,S)$ is indeed strongly étale.
 It is clear by \cite{Hu96}, Corollary~1.7.3~(iii) that it is étale.
 For a point $x = (\supp x,k(x)^+,c_x) \in \Spa(X,S)$ the morphism $c_x: \Spec k(x)^+ \to S$ uniquely lifts to $c_{x,Y}: \Spec k(x)^+ \to Y$ by the valuative criterion for properness.
 The base change
 \[
  \Spec k(x)^+ \times_Y Y' \to \Spec k(x)^+
 \]
 of $Y' \to Y$ is étale.
 The preimages of $x$ correspond to the local rings of $\Spec k(x)^+ \times_Y Y'$ at points over the closed point of~$\Spec k(x)^+$.
 These are all unramified, hence $\Spa(Y'\times_Y X,Y') \to \Spa(X,S)$ is strongly étale.
 
 It is clear that $\Spa(Y'\times_Y X,Y') \to \Spa(X,S)$ is a Riemann-Zariski morphism as
 \[
  \begin{tikzcd}
   Y' \times_Y X	\ar[r]	\ar[d]	& X	\ar[d]	\\
   Y'				\ar[r]			& Y
  \end{tikzcd}
 \]
 is Cartesian and thus has universally closed diagonal (see \cite{HueSch20}, Lemma~12.7).

 Now we show that coverings are mapped to coverings.
 Let $(Y'_i \to Y')_{i \in I}$ be a covering in $Y_\et$.
 Given a point $y \in \Spa(Y' \times_Y X,Y')$ we have to find $i \in I$ and $y_i \in \Spa(Y'_i \times_Y X,Y'_i)$ mapping to $y$.
 Let $s \in Y'$ be the center of~$k(y)^+$ in~$Y'$, i.e. the image of the closed point of $\Spec k(y)^+$ under the map $c_y:\Spec k(y)^+ \to Y'$.
 There is $i \in I$ and $s_i \in Y'_i$ mapping to~$s$.
 Consider the fiber product
 \[
  \begin{tikzcd}
   \Spec k(y)^+ \times_{Y'} Y'_i	\ar[r]	\ar[d,"\text{étale}"']	& Y'_i	\ar[d,"\text{étale}"]	\\
   \Spec k(y)^+						\ar[r,"\phi"]					& Y'.
  \end{tikzcd}
 \]
 There is a point of $\Spec k(y)^+ \times_{Y'} Y'_i$ lying over the closed point of $\Spec k(y)^+$ and mapping to $s_i \in Y'_i$.
 Its local ring~$R_i$ is a valuation ring which is an unramified extension of~$k(y)^+$.
 Restricted to the generic point $\Spec k(y)$ of $\Spec k(y)^+$, the map~$c_y$ factors through the inclusion $\Spec k(y) \hookrightarrow Y' \times_Y X$ (with trivial residue field extension).
 We obtain a factorization
 \[
  \Spec k(y) \times_{Y'} Y'_i \hookrightarrow X \times_Y Y'_i \to Y'_i,
 \]
 also with trivial residue field extensions.
 The generic point of~$\Spec R_i$ is one of the points of $\Spec k(y) \times_{Y'} Y_i$.
 We denote its image in $Y'_i \times_Y X$ by $y_i$.
 Writing $c_i$ for the map $\Spec R_i \to Y'_i$ we obtain a preimage $(y_i,R_i,\phi_i)$ of $y$.
\end{proof}

Our goal is to show that the morphisms of sites $\Psi_Y$ induce an equivalence of topoi in the limit over all modifications~$Y$.
The main results we need for this are provided by the following two lemmas.

\begin{lemma} \label{refine_cover_model}
 Let $\cU \to \cX = \Spa(X,S)$ be a strongly étale Riemann-Zariski morphism.
 Then there exists an open covering of~$\cU$ by subsets of the form $\Psi_Y^{-1}(Y')$ for some $X$-modification~$Y$ of~$S$ and $Y' \to Y$ étale.
 \[
  \begin{tikzcd}
   Y' \times_Y X	\ar[r]						\ar[d]	& Y'	\ar[d,"\text{étale}"]	\\
   X				\ar[r,"\text{dominant}"]	\ar[dr]	& Y		\ar[d,"\text{proper}"]	\\
   														& S.
  \end{tikzcd}
 \]
\end{lemma}

\begin{proof}
 We may assume that~$\cU$ is of the form $\Spa(U,T)$ coming from a diagram
 \[
  \begin{tikzcd}
   U	\ar[r,"\text{étale}"]	\ar[d]	& X	\ar[d]	\\
   T	\ar[r,"\text{f.t.}"]			& S
  \end{tikzcd}
 \]
 with universally closed diagonal.
 Let $u \in \cU = \Spa(U,T)$ be a Riemann-Zariski point with image $x \in \cX = \Spa(X,S)$.
 Since $\cU \to \cX$ is Riemann-Zariski, the diagram
 \[
  \begin{tikzcd}
   \Spec \cO_{\cU,u}	\ar[r]	\ar[d]	& \Spec \cO_{\cX,x}	\ar[d]	\\
   \Spec \cO^+_{\cU,u}	\ar[r]			& \Spec \cO^+_{\cX,x}.
  \end{tikzcd}
 \]
 has universally closed diagonal.
 But $\cO_{\cU,u}$ and $\cO^+_{\cU,u} \otimes_{\cO^+_{\cX,x}} \cO_{\cX,x}$ are both localizations of~$\cO^+_{\cU,u}$, which implies that the diagram is even Cartesian.
 The morphism
 \[
  \Spa(\cO_{\cU,u},\cO^+_{\cU,u}) \to \Spa(\cO_{\cX,x},\cO^+_{\cX,x})
 \]
 is a localization of the strongly étale morphism
 \[
  \cU_x: = \cU \times_\cX \Spa(\cO_{\cX,x},\cO^+_{\cX,x}) \to \Spa(\cO_{\cX,x},\cO^+_{\cX,x}).
 \]
 The local rings of $\cO_{\cU_x}^+$ are obtained by ring theoretic localization of its global sections as the latter are already semivaluation rings (\cite{HueAd}, Lemma~11.13).
 Using \cite{HueAd}, Lemma~6.6, we conclude that $\cO^+_{\cU,u}$ is a local ring of an étale $\cO^+_{\cX,x}$-algebra.
 
 By \cite{Tem11} we can write $\cO^+_{\cX,x}$ as a colimit
 \[
  \cO^+_{\cX,x} \cong \colim_Y \cO_{Y,y}
 \]
 running over all $X$-modifications~$Y$ of~$S$ where~$y$ denotes the center of~$x$ in~$Y$.
 Hence, we find a modification~$Y$ and an étale morphism $Y' \to Y$ such that~$\cO^+_{\cU,u}$ is a local ring of
 \[
  Y' \times_Y \Spec \cO^+_{\cX,x}.
 \]
 Replacing~$Y'$ by an open subscheme, we achieve that
 \[
  \cU_u := \Spa(Y' \times_Y X,Y') \to \Spa(X,S)
 \]
 factors through a neighborhood of~$u$ in $\cU = \Spa(U,T)$.
 The spaces~$\cU_u$ for varying~$u$ cover~$\cU$ and provide the required refinement.
\end{proof}

\begin{lemma} \label{fully_faithful}
 Let $Y \to S$ be an $X$-modification such that $Y$ is relatively normal in~$X$.
 Then $\Psi_Y^{-1}$ is fully faithful.
\end{lemma}

\begin{proof}
 Take two objects $Y'_1$ and $Y'_2$ in $Y_\et$.
 We set for $i=1,2$
 \[
  X'_i = Y'_i \times_Y X.
 \]
 Suppose we have two morphisms $a,b: Y'_1 \to Y'_2$ that induce the same morphism
 \[
  \Spa(X'_1,Y'_1) \longrightarrow \Spa(X'_2, Y'_2).
 \]
 Then
 \[
  a \times_Y X = b \times_Y X : X'_1 \longrightarrow X'_2.
 \]
 Since $X'_1 \to Y'_1$ and $X'_2 \to Y'_2$ are dominant, we conclude by continuity that $a=b$ as morphisms of topological spaces.
 It remains to show that the homomorphisms on the structure sheaves defined by~$a$ and~$b$ are the same.
 We take affine open subsets $U'_1$ and $U'_2$ of $Y'_1$ and $Y'_2$, respectively, such that $a(U'_1) \subseteq U'_2$.
 In the diagram
 \[
  \begin{tikzcd}[column sep=3cm]
   \cO_{X'_1}(X'_1)							& \cO_{X'_2}(X'_2)	\ar[l,"(a \times_Y X)^\#= (b \times_Y X)^\#"']	\\
   \cO_{Y'_1}(Y'_1)							\ar[u,hookrightarrow]	& \cO_{Y'_2}(Y'_2)							\ar[l,shift left=1.5,"b^\#"]	\ar[l,shift right=1.5,"a^\#"']	\ar[u,hookrightarrow]
  \end{tikzcd}
 \]
 the vertical arrows are injective because $X'_i \to Y'_i$ is dominant and $Y'_i$ is affine.
 We conclude that $a^\#=b^\#$ and thus $a=b$.
 
 Now suppose we are given a morphism
 \[
  g: \Spa(X'_1,Y'_1) \longrightarrow \Spa(X'_2, Y'_2).
 \]
 over $\Spa(X,S) = \Spa(X,Y)$.
 It induces a morphism
 \[
  f: X'_1 \to X'_2
 \]
 over~$X$ of support schemes.
 We consider the graph
 \[
  \Gamma_f \subseteq X'_1 \times_X X'_2
 \]
 of~$f$ and its scheme theoretic image
 \[
  \bar{\Gamma}_f \subseteq Y'_1 \times_Y Y'_2.
 \]
 We have a diagram
 \[
  \begin{tikzcd}
   \Gamma_f			\ar[r,closed]	\ar[d,"\text{dominant}"']	& X'_1 \times_X X'_2	\ar[d]	\\
   \bar{\Gamma}_f	\ar[r,closed]	\ar[d,"\pi"']				& Y'_1 \times_Y Y'_2	\ar[dl,"\text{étale}"]	\ar[dr,"\text{étale}"]	\\
   Y'_1				\ar[dr,"\text{étale}"']						&																		& Y'_2	\ar[dl,"\text{étale}"]	\\
   																& Y
  \end{tikzcd}
 \]
 and we want to show that~$\pi$ is an isomorphism.
 We know it is quasifinite as a composition of a closed immersion with an étale morphism.
 
 Moreover, $X'_1 \to Y'_1$ is a basechange of $X \to Y$ along an étale morphism.
 As $Y$ is relatively normal in~$X$, this implies that $Y'_1$ is relatively normal in~$X'_1$ (see \cite[Tag 03GV]{stacks-project}).
 But $\Gamma_f \cong X'_1$ via the first projection, so $Y'_1$ is relatively normal in~$\Gamma_f$.
 If we show that~$\pi$ is proper (hence finite), we are done as we would then have a factorization
 \[
  \Gamma_f \to \bar{\Gamma}_f \overset{\pi}{\longrightarrow} Y'_1
 \]
 of $\Gamma_f \to Y'_1$ into a dominant morphism followed by a finite morphism.
 Since $Y'_1$ is relatively normal in~$\Gamma_f$, this forces~$\pi$ to be an isomorphism.
 
 In order to show that~$\pi$ is proper we verify the valuative criterion in the form stated in \cref{valuative_criterion} below.
 Let $x_1$ be a point of $X'_1 \cong \Gamma_f$ and $R_1$ a valuation ring of $k(x_1)$ fitting into the solid arrow diagram
 \[
  \begin{tikzcd}
   \Spec k(x_1)	\ar[r]	\ar[dd]					& \Gamma_f \cong X'_1		\ar[d]	\\
   												& \bar{\Gamma}_f			\ar[d]	\\
   \Spec R_1	\ar[r,"\phi_1"]	\ar[ur,dashed]	& Y'_1.
  \end{tikzcd}
 \]
 We have to show that there exists a unique dotted arrow making the diagram commute.
 The triple $(x_1,R_1,\phi_1)$ defines a point of $\Spa(X'_1,Y'_1)$.
 We set $(x_2,R_2,\phi_2) := g(x_1,R_1,\phi_1)$.
 Then $x_2= f(x_1)$ and $R_2 = R_1|_{k(x_2)}$.
 As $g$ is defined over $\Spa(X,Y)$, we have a commutative diagram
 \[
  \begin{tikzcd}
   \Spec R_1	\ar[rr]	\ar[d,"\phi_1"']	&		& \Spec R_2	\ar[d,"\phi_2"]	\\
   Y'_1			\ar[dr]						&		& Y'_2		\ar[dl]		\\
   											& Y.
  \end{tikzcd}
 \]
 This defines a morphism
 \[
  \Spec R_1 \longrightarrow Y'_1 \times_Y Y'_2.
 \]
 Restricted to the generic point, this morphism factors as
 \[
  \Spec k(x_1) \longrightarrow \Gamma_f \hookrightarrow X'_1 \times_X X'_2 \longrightarrow Y'_1 \times_Y Y'_2.
 \]
 Therefore there is a unique factorization
 \[
  \Spec R_1 \longrightarrow \bar{\Gamma}_f \longrightarrow Y'_1 \times_Y Y'_2
 \]
 by the definition of~$\bar{\Gamma}_f$ as the scheme theoretic image of~$\Gamma_f$.
 This produces the dotted arrow we were looking for. 
\end{proof}

\begin{lemma} \label{valuative_criterion}
 Let $f: T \to Z$ and $g:V \to T$ be morphisms of schemes with~$g$ dominant.
 Then $f$ satisfies the valuative criterion for properness if and only if for every $v \in V$  and for every valuation ring $R$ of $k(v)$ fitting into a commutative diagram
 \[
  \begin{tikzcd}
   \Spec k(v)	\ar[r,hookrightarrow]	\ar[dd,hookrightarrow]	& V	\ar[d,"g"]	\\
   																& T	\ar[d,"f"]	\\
   \Spec R		\ar[r]					\ar[ur,dashed]			& Z
  \end{tikzcd}
 \]
 the dotted arrow exists and is unique.
\end{lemma}

We now consider the limit of the morphisms of sites~$\Psi_Y$:
\[
 \Psi : \Spa(X,S)_{\set,\RZ} \longrightarrow \lim_Y Y_\et
\]

\begin{proposition} \label{psi_equivalence}
 The morphism of sites~$\Psi$ induces an equivalence of topoi.
\end{proposition}

\begin{proof}
 By \cite{SGA4}, III 4.1 we need to show the following:
 \begin{enumerate}[(a)]
  \item $\lim_Y Y_\et$ has finite limits and~$\Psi^{-1}$ commutes with these;
  \item Every object $V$ of $\Spa(X,S)_{\set,\RZ}$ has a covering by objects of the form $\Psi^{-1}(U)$;
  \item A family $(U_i \to U)$ in $\lim_Y Y_\et$ is a covering if and only if $(\Psi^{-1}(U_i) \to \Psi^{-1}(U))$ is a covering.
  \item $\Psi^{-1}$ is fully faithful.
 \end{enumerate}
 Let us show (a).
 We know that $Y_\et$ has finite limits for every $X$-modification~$Y$ of~$S$ as this is true for any scheme.
 The objects of $\lim_Y Y_\et$ come from a finite level and as in a finite limit only finitely many objects are involved, the limit can be computed at a finite level.
 So finite limits exist in $\lim_Y Y_\et$.
 It is clear from the construction that $\Psi^{-1}$ commutes with fiber products.
 We are left with showing that it commutes with coequalizers.
 Let $Y$ be an $X$-modification of~$S$ and let $Y'$ be the equalizer of a diagram
 \[
  \begin{tikzcd}
   (Y'_1	\ar[r,shift left=1.5,"a"]	\ar[r,shift right=1.5,"b"']	& Y'_2)
  \end{tikzcd}
 \]
 in $Y_\et$.
 Then
 \[
  Y' = Y'_1 \times_{Y'_1 \times_Y Y'_1} (Y'_1 \times_{Y'_2} Y'_1),
 \]
 where the fiber product $Y'_1 \times_{Y'_2} Y'_1$ is taken with respect to the maps~$a$ and~$b$ and $Y'_1 \to Y'_1 \times_Y Y'_1$ is the diagonal.
 Since~$\Psi^{-1}$ commutes with fiber products, we are done with checking condition~(a).

 Condition~(b) is satisfied by \cref{refine_cover_model}.
 Let us now treat~(c).
 We have already verified in \cref{psi_morphism_sites} that $\Psi_Y$ (and hence $\Psi$) is a morphism of sites.
 So $\Psi^{-1}$ maps coverings to coverings.
 For the converse take a family of morphisms $(Y'_i \to Y')_{i \in I}$ in $Y_\et$ for some $X$-modification~$Y$ of~$S$ and assume that
 \[
  (\Spa(Y'_i \times_Y X,Y'_i) \to \Spa(Y' \times_Y X,Y'))_{i \in I}
 \]
 is a strongly étale covering.
 Take $y' \in Y'$.
 Since $X \to Y$ is dominant and this property is preserved by flat base change, we find a point $x \in Y' \times_Y X$ whose image in~$Y'$ is a generalization of~$y'$.
 Choosing a valuation ring~$R$ of $k(x)$ with center $y'$ we obtain a point $(x,R,\phi)$ of $\Spa(Y' \times_Y X,Y')$.
 The center of any preimage of $(x,R,\phi)$ in some $\Spa(Y'_i \times_Y X,Y'_i)$ gives a preimage of~$y'$ in $Y'_i$.
 
 Finally, (d) follows from \cref{fully_faithful} noting that we may always replace an $X$-modifi\-cation~$Y$ of~$S$ by its integral closure in~$X$.
\end{proof}

\begin{corollary} \label{cohomology_models}
 Let $X \to S$ be a separated morphism of schemes and~$\mathcal{F}$ a torsion sheaf on $\Spa(X,S)_\set$.
 For an $X$-modification~$Y$ of $S$ we denote by~$\varphi_Y$ the composition of the projection
 \[
  \pi : \Spa(X,S)_\set \longrightarrow \Spa(X,S)_{\set,\RZ}
 \]
 with the morphism of sites
 \[
  \Psi_Y: \Spa(X,S)_{\set,\RZ} \to Y_\et
 \]
 studied above.
 Then we obtain natural isomorphisms
 \[
  H^i(\Spa(X,S)_\set,\mathcal{F}) \cong \colim_Y H^i(Y_\et,\varphi_{Y,*}\mathcal{F})
 \]
\end{corollary}

\begin{proof}
 We know that $\pi_*$ is exact by \cref{pi_*_exact} and that the limit of $\Psi_{Y,*}$ over all compactifications~$Y$ is exact because it is an equivalence of categories by \cref{psi_equivalence}.
 Therefore, $\varphi_{Y,*}$ is exact and the result follows from the Leray spectral sequence.
\end{proof}

We need a version of \cref{cohomology_models} for pseudoadic spaces.
We fix an adic space of the form $\cX = \Spa(X,S)$ for a morphism of schemes $X \to S$.
Then we consider the pseudoadic space defined by a closed subset $\cZ \subseteq \Spa(X,S)$.

\begin{lemma} \label{image_specialization_closed}
 Let~$Y$ be an $X$-modification of~$S$.
 The image~$Z_Y$ of~$\cZ$ in~$Y$ is closed.
\end{lemma}

\begin{proof}
 The center map
 \[
  c : \Spa(X,S) \longrightarrow Y
 \]
 is continuous and spectral.
 Therefore, it maps the proconstructible subset~$\cZ$ of $\Spa(X,S)$ to a proconstructible subset of~$Y$.
 So~$Z_Y$ is proconstructible.
 Moreover, it is closed under specialization:
 If $y_1 \in Z_Y$ specializes to $y_2 \in Y$, we can find a valuation~$v$ of $k(y_1)$ with center~$y_2$.
 Furthermore, we find $z \in \cZ$ with center~$y_1$.
 Lifting~$v$ to some valuation~$v'$ of the specialization field $k(z)^\succ \supseteq k(y_1)$, we can compose~$z$ with~$v'$ to obtain a point $z' := v' \circ z \in \cZ$ with center~$y_2$.
\end{proof}

For every $X$-modification~$Y$ of~$S$ we denote by~$Z_Y$ the image of~$\cZ$ in~$Y$.
It is a closed subset by \cref{image_specialization_closed}.
In the diagram
\[
 \begin{tikzcd}
  \cZ	\ar[r,"\varphi_{Z_Y}"]	\ar[d,closed,"\iota"']	& Z_Y	\ar[d,closed,"\iota_Y"]	\\
  \cX	\ar[r,"\varphi_Y"]								& Y
 \end{tikzcd}
\]
the vertical map~$\iota$ is a morphism of pseudoadic spaces.
Hence, it induces a morphism of sites $\cZ_\set \to \cX_\set$.
Similarly, the morphism of schemes~$\iota_Y$ induces a morphism of étale sites $Z_{Y,\et} \to Y_\et$.
Moreover, $\varphi_Y$ induces a morphism of sites $\cX_\set \to Y_\et$ that sends an étale map $U \to Y$ to $\Spa(X \times_Y U, U)$, which is strongly étale over $\cX = \Spa(X,S)$.
One might think that in the same way~$\varphi_{Z_Y}$ induces a morphism of sites $\cZ_\set \to Z_{Y,\et}$.
But this is not so obvious as by definition the strongly étale morphisms to~$\cZ_\set$ are étale morphisms to $\cX = \Spa(X,S)$ (with unramified residue field extensions over~$\cZ$).
However, at the level of topoi we do get a morphism.

\begin{lemma} \label{specialization_map_topoi}
 There is a unique morphism of topoi
 \[
  \varphi_{Z_Y} : \cZ_\set \longrightarrow Z_{Y,\et}
 \]
 making the diagram
 \[
  \begin{tikzcd}
   \cZ_\set	\ar[r,"\varphi_{Z_Y,*}"]	\ar[d,"\iota_*"]	& Z_{Y,\et}	\ar[d,"\iota_{Y,*}"]	\\
   \cX_\set	\ar[r,"\varphi_{Y,*}"]							& Y_\et
  \end{tikzcd}
 \]
 commute.
\end{lemma}

\begin{proof}
 Since $\iota_Y : Z_Y \hookrightarrow Y$ is a closed immersion, $Z_{Y,\et} \to Y_\et$ is an immersion of topoi.
 By \cite{SGA4}, Exposé~IV, Proposition~9.1.4 it thus suffices to show that the essential image of $(\varphi_Y \circ \iota)_*$ lies in the essential image of~$\iota_{Y,*}$.
 In order to show this we take an étale morphism $U \to Y$ whose image is disjoint from~$Z_Y$ and verify that for a sheaf~$\cF$ on~$\cZ_\set$ we have
 \[
  (\varphi_Y \circ \iota)_* \cF(U) = \{*\}.
 \]
 The left hand side equals
 \[
  \cF(\Spa(U \times_Y X,U),\varnothing).
 \]
 since the image of~$U$ in~$Y$ is disjoint from~$Z_Y$, so no point of~$\cZ$ lifts to $\Spa(U \times_Y X,U)$.
 This means that $\varnothing \to (\Spa(U \times_Y X,U),\varnothing)$ is a covering in~$\cZ_\set$ and thus
 \[
  \cF(\Spa(U \times_Y X,U),\varnothing) = \cF(\varnothing) = \{*\}.
 \]
\end{proof}

\begin{corollary} \label{pseudoadic_colim_models}
 Let $X \to S$ be a morphism of schemes and $\cZ \subseteq \Spa(X,S)$ a closed subset.
 For an $X$-modification~$Y$ of~$S$ we write~$Z_Y$ for the image of~$\cZ$ in~$Y$ and
 \[
  \varphi_{Z_Y} \colon \Shv(\cZ_\set) \longrightarrow \Shv(Z_{Y,\et})
 \]
 for the morphism of topoi from \cref{specialization_map_topoi}.
 Then for every abelian sheaf~$\cG$ on~$\cZ_\set$
 \[
  H^i(\cZ_\set,\cG) \cong \colim_Y H^i(Z_{Y,\et},\varphi_{Z_Y,*} \cG).
 \]
 Moreover, the statement is also true for any sheaf of groups if $i \le 1$ and for any shef of sets for $i=0$.
\end{corollary}

\begin{proof}
 We consider the commutative square of morphisms of topoi
 \[
  \begin{tikzcd}
   \Shv(\cZ_\set)	\ar[r,"\varphi_{Z_Y}"]	\ar[d,hook,"iota"']	& \Shv(Z_{Y,\et})	\ar[d,"\iota_Y"]	\\
   \Shv(\cX_\set)	\ar[r,"\varphi_Y"']							& \Shv(Y_\et)
  \end{tikzcd}
 \]
 from \cref{specialization_map_topoi}.
 Then
 \[
  H^i(\cZ_\set,\cG) \cong H^i(\Spa(X,S),\iota_* \cG)
 \]
 and
 \[
  H^i(Z_{Y,\et},\varphi_{Z_Y,*}\cG) \cong H^i(Y_\et,\iota_{Y,*}\varphi_{Z_Y,*}\cG) \cong H^i(Y_\et,\varphi_{Y,*}\iota_*\cG).
 \]
 Therefore, the assertion follows from \cref{cohomology_models} applied to~$\iota_*\cG$.
\end{proof}

\section{Proper morphisms of discretely ringed adic spaces} \label{sect_proper}

In this section we want to understand the proper adic spaces over a discretely ringed adic space $\cS = \Spa(S,S^+)$ for a morphism of schemes $S \to S^+$.
It will turn out that they are all of the form $\Spa(X,S^+)$ for a proper scheme~$X$ over~$S$, that we call support scheme.
The construction of the support scheme is not limited to proper spaces over~$\cS$ but quite general.

\begin{lemma} \label{support_scheme}
 Let~$\cX$ be a discretely ringed adic space.
 Then there is a scheme~$X$ and an open quotient map (denoted \emph{support morphism}) of locally ringed spaces
 \[
  \supp : (\cX,\cO_\cX) \longrightarrow (X,\cO_X)
 \]
 whose fibers are the sets of points with a common vertical generalization and such that $\cO_X \to \supp_*\cO_\cX$ is an isomorphism.
\end{lemma}

\begin{proof}
 For affinoid adic spaces, i.e. $\cX = \Spa(A,A^+)$, the support morphism is the map
 \[
  \supp \colon \Spa(A,A^+) \longrightarrow \Spec A
 \]
 that maps a point $x \in \cX$ to the prime ideal
 \[
  \supp(x) = \{ a \in A \mid \lvert a(x) \rvert = 0\}
 \]
 of~$A$.
 By construction, this morphism is surjective and identifies points with a common vertical generalization.
 Moreover, the map is surjective and maps a rational subset $\{ x \in \cX \mid \lvert f_i(x) \rvert \le \lvert g(x) \rvert \ne 0\}$ to the open subset $\{\p \in \Spec A \mid g \notin \p\}$.
 Hence, it is an open quotient map.
 Note that indeed the pushforward of the structure sheaf of $\Spa(A,A^+)$ identifies with the structure sheaf on $\Spec A$.
 
 In the general case we need to convince ourselves that the affine construction glues.
 This is the case because glueing data for affinoid opens in~$\cX$ induce scheme theoretic glueing data on the spport.
\end{proof}

It is clear that the support map is functorial in the following sense.
A morphism $\cX \to \cS$ of discretely ringed adic spaces induces a morphism $X \to S$ of support schemes fitting into the commutative diagram
\[
 \begin{tikzcd}
  \cX	\ar[r]	\ar[d,"\supp"']	& \cS 	\ar[d,"\supp"]	\\
  X		\ar[r]					& S.
 \end{tikzcd}
\]
If~$\cX$ is of the form $\Spa(X,X^+)$ for a morphism of schemes $X \to X^+$, then $\supp(\cX) = X$ and the support map sends a triple $(x,k(x)^+,c_x) \in \cX$ to $x \in X$.
In particular, if~$\cX$ is a discretely ringed adic space over $\cS = \Spa(S,S^+)$, then its support~$X$ is a scheme over~$S$.
Moreover $\supp$ commutes with fiber products and cofiltered limits.

Recall that a morphism of adic spaces $\cX \to \cS$ is \emph{proper} if it is of $+$weakly finite type, separated, and universally specializing.

\begin{lemma} \label{support_properties}
 Let $S \to S^+$ be a morphism of schemes and~$\cX$ a discretely ringed adic space over $\cS = \Spa(S,S^+)$.
 Suppose that $\cX \to \cS$ has one of the following properties:
 \begin{enumerate}[(a)]
  \item	quasicompact,
  \item	(locally) of weakly finite type,
  \item	separated,
  \item	proper.
 \end{enumerate}
 Then $X \to S$ has the same property (without the adverb ``weakly'' for property (b)).
\end{lemma}

\begin{proof}
 (a). If the preimage of an affinoid $\Spa(R,R^+) \subseteq \cS$ has a cover by finitely may affinoids $\Spa(A_i,A_i^+) \subseteq \cX$, then the preimage of $\Spec R \subseteq S$ is covered by the affines $\Spec A_i$.
 Moreover~$S$ is covered by affines $\Spec R$ arising in this way.
 
 (b). Let $x \in X$ be a point.
 We pick a preimage $\underline{x} \in \cX$ under the support map.
 By assumption $\underline{x}$ has an open affinoid neighborhood $\Spa(A,A^+)$ mapping into an affinoid open $\Spa(R,R^+) \subseteq \cS$ such that~$A$ is of finite type over~$R$.
 But then $\Spec A$ is an affine neighborhood of~$x$ mapping into $\Spec R \subseteq S$.
 
 (c). We consider the diagram
 \[
  \begin{tikzcd}
   \cX	\ar[r]	\ar[d,"\supp"']		\ar[r]	& \cX \times_\cS \cX	\ar[d,"\supp"']	\\
   X	\ar[r]								& X \times_S X.
  \end{tikzcd}
 \]
 Let $x_1 \in X \times_S X$ be in the image of the diagonal and $x_2 \in X \times_S X$ a specialization of~$x_1$.
 Then the trivial valuations~$x_1^\triv$ and~$x_2^\triv$ on~$k(x_1)$ and~$k(x_2)$, respectively, define points of $\cX \times_\cS \cX$ such that~$x_2^\triv$ is a specialization of~$x_1^\triv$ and~$x_1^\triv$ is contained in the diagonal.
 Since $\cX \to \cS$ is separated, it follows that~$x_2^\triv$ is contained in the diagonal as well.
 Then also~$x_2$ is contained in the diagonal and consequently $X \to S$ is separated.
 
 (d). By~(b) and~(c) the morphism $X \to S$ is of finite type and separated.
 It remains to show that it is universally closed.
 We consider the base change
 \[
  X' \colon = X \times_S S' \longrightarrow S'
 \]
 of $X \to S$ along a morphism of schemes $S' \to S$.
 It induces morphisms of adic spaces
 \[
  \cS' \colon = \Spa(S',S^+) \longrightarrow \Spa(S,S^+)
 \]
 and
 \[
  \cX' \colon = \cX \times_\cS \cS' \longrightarrow \cS'
 \]
 fitting into the commutative diagram
 \[
  \begin{tikzcd}
   \cX'	\ar[r]	\ar[d,"\supp"']	& \cS'	\ar[d,"\supp"]	\\
   X'	\ar[r]					& S'.
  \end{tikzcd}
 \]
 The upper horizontal map is closed.
 Since the support map is a quotient map, the same holds true for the lower horizontal map.
\end{proof}

\begin{lemma} \label{vertical_compactification}
 Let $S \to S^+$ be a morphism of schemes and~$\cX$ a discretely ringed adic space separated over $\cS = \Spa(S,S^+)$ with support~$X$.
 Then the map
 \begin{align*}
  \iota \colon \cX	& \longrightarrow \Spa(X,S^+),	\\
  x					& \longmapsto (\supp(x),k(x)^+,c_x)),
 \end{align*}
 where $s \in \cS$ is the image of~$x$ and~$c_x$ denotes the composition
 \[
  c_x \colon \Spec k(x)^+ \longrightarrow \Spec k(s)^+ \overset{c_s}{\longrightarrow} S^+,
 \]
 defines a pro-open immersion of adic spaces over~$\cS$ with dense image.
\end{lemma}

\begin{proof}
 Note that~$\iota$ is indeed defined over~$\cS$.
 By construction the image of the point $(\supp(x),k(x)^+,c_x)$ in~$\cS$ is $(\supp(s),k(s)^+,c_s) = s$.
 
 Let us convince ourselves that~$\iota$ is injective.
 Suppose that~$x_1$ and~$x_2$ are two points of~$\cX$ with $\iota(x_1) = \iota(x_2)$.
 Both points map to the same point $s \in \cS$.
 Moreover,
 \[
  (k(x_1),k(x_1)^+) = (k(x_2),k(x_2)^+) =: (k,k^+).
 \]
 We get a commutative square
 \[
  \begin{tikzcd}
   \Spa(k,k^+)	\ar[r,"x_1"]	\ar[d,"x_2"']	& \cX	\ar[d]	\\
   \cX			\ar[r]							& \cS
  \end{tikzcd}
 \]
 and thus a morphism
 \[
  \Spa(k,k^+) \longrightarrow \cX \times_\cS \cX.
 \]
 Both~$x_1$ and~$x_2$ generalize to the trivial valuation on~$k$.
 Since $\supp(x_1) = \supp(x_2)$, the two resulting maps $\Spa(k,k) \to \cX$ are the same and thus the generic point of $\Spa(k,k^+)$ maps to the diagonal in $\cX \times_\cS \cX$.
 Since~$\cX$ is separated over~$\cS$, the diagonal is closed and hence $x_1 = x_2$.

 Locally on an affinoid open $\Spa(A,A^+)$ of~$\cX$ mapping into an affinoid open $\Spa(R,R^+)$ of $\cS$ the map~$\iota$ is defined by the homomorphism of Huber pairs
 \[
  (A,R^+) \longrightarrow (A,A^+)
 \]
 induced by the identity on~$A$, so~$\iota$ defines a morphism of adic spaces.
 Here we also see that~$\iota$ is a pro-open immersion.
 Indeed, the adic spectrum of the above map of Huber pairs is the inverse limit of the open immersions
 \[
  \Spa(A,B^+) \longrightarrow \Spa(A,R^+),
 \]
 where~$B^+$ ranges over the finitely generated $R^+$-subalgebras of~$A^+$.
 It is also clear that~$\iota$ has dense image.
 \end{proof}


\begin{lemma} \label{Spa_properties}
 Let $S \to S^+$ and $X \to S$ be morphisms of schemes and set
 \[
  \cS = \Spa(S,S^+), \qquad \cX = \Spa(X,S^+).
 \]
 We consider the properties
 \begin{enumerate}[(a)]
  \item	quasicompact,
  \item (locally) of ${}^+$weakly finite type,
  \item	separated,
  \item	proper.
 \end{enumerate}
 Then $X \to S$ has one of these properties (without ``${}^+$-weakly'' for property~(b)) if and only if $\cX \to \cS$ has the respective property.
\end{lemma}

\begin{proof}
 Note that~$X$ is the support scheme of~$\cX$.
 Therefore, the ``if'' part of the assertion follows from \cref{support_properties}.
 
 (a). Suppose that $X \to S$ is quasicompact.
 We take an affinoid open $\Spa(R,R^+) \subseteq \cS$ coming from a commutative square
 \[
  \begin{tikzcd}
   \Spec R		\ar[r]	\ar[d]	& S	\ar[d]	\\
   \Spec R^+	\ar[r]			& S^+.
  \end{tikzcd}
 \]
 The preimage of $\Spa(R,R^+)$ in $\cX$ equals $\Spa(X \times_S \Spec R,\Spec R^+)$.
 By assumption $X \times_S \Spec R$ is quasicompact, hence also $\Spa(X \times_S \Spec R,R^+)$.
 
 (b). Let $x \in \cX$ be a point.
 In order to show that $\cX \to \cS$ is locally of ${}^+$-weakly finite type, we may assume that~$S^+$ is affine, $S^+ = \Spec R^+$.
 If $X \to S$ is locally of finite type, there is an affine open neighborhood $\Spec A$ of $\supp(x)$ mapping into an affine open $\Spec R \subseteq S$ such that~$A$ is finitely generated over~$R$.
 Then $\Spa(A,R^+)$ is an affinoid neighborhood of~$x$ mapping into $\Spa(R,R^+)$ showing that $\cX \to \cS$ is locally of ${}^+$weakly finite type.
 
 Combining~(a) with what we have just proved, we obtain that $\cX \to \cS$ is of ${}^+$weakly finite type if $X \to S$ is of finite type.
 
 (c). Note that
 \[
  \cX \times_\cS \cX = \Spa(X \times_S X,S^+).
 \]
 We have to show that~$\cX$ is closed in $\cX \times_\cS \cX$ if~$X$ is closed in $X \times_S X$.
 For a point $x \in (\cX \times_\cS \cX) \setminus \cX$ there is an open neighborhood~$U$ of $\supp(x)$ in $X \times_S X$ disjoint from the diagonal.
 Since a point in $\cX \times \cS \cX$ is contained in~$\cX$ if and only if its support is contained in~$X$, the open neighborhood $\Spa(U,S^+)$ of~$x$ is disjoint from the diagonal.
 
 (d). Assume that $X \to S$ is proper.
 By~(b) and~(c) we know that $\cX \to \cS$ is of finite type and separated.
 It remains to show it is universally specializing.
 Pick a morphism $\cS' \to \cS$.
 We may assume that~$\cS'$ is of the form $\Spa(S',S'^+)$ for a commutative square
 \[
  \begin{tikzcd}
   S'	\ar[r]	\ar[d]	& S	\ar[d]	\\
   S'^+	\ar[r]			& S^+
  \end{tikzcd}
 \]
 because the statement is local on~$\cS'$.
 Set $X' = X \times_S S'$ and
 \[
  \cX' = \cX \times_\cS \cS' = \Spa(X',S'^+).
 \]
 We pick a point $x' = (\supp(x'),k(x')^+,c_{x'})$ in~$\cX'$ and assume that its image $s' = (\supp(s'),k(s')^+,c_{s'})$ in~$\cS'$ has a specialization $s_1' \in \cS'$.
 We have to lift~$s'_1$ to a specialization of~$x'$.
 We can treat vertical and horizontal specializations separately.
 Let us first assume that~$s'_1$ is a vertical specialization of~$s'$.
 Then $\supp(s') = \supp(s'_1)$ and
 \[
  k(s'_1)^+ \subseteq k(s')^+ \subseteq k(s').
 \]
 In fact, $k(s'_1)^+$ is the preimage in $k(s')^+$ of the valuation ring $k(s'_1)^+/\m_{k(s')^+}$ of $k(s')^\succ = k(s')^+/\m_{k(s')}^+$.
 Using Chevalley's extension theorem we find a valuation ring~$R'_1$ of $k(x')$ extending $k(s'_1)^+$ and such that
 \[
  R'_1 \subseteq k(x')^+ \subseteq k(x').
 \]
 More precisely, we extend $k(s'_1)^+/\m_{k(s')^+}$ from $k(s')^\succ$ to $k(x')^\succ$ using \cite{EP2005}, Theorem~3.1.2.
 Then~$R'_1$ is the preimage of this valuation ring in $k(x')^+$.
 Defining the homomorphism~$c'_1$ as the composition
 \[
  c'_1 : \Spec R'_1 \longrightarrow \Spec k(s'_1) \longrightarrow S^+
 \]
 we obtain a point
 \[
  x'_1 = (\supp x',R'_1,c'_1) \in \cX'
 \]
 lying over~$s'_1$ and specializing~$x'$.
 
 Now assume that $s' \rightsquigarrow s'_1$ is a horizontal specialization.
 This means that there is a valuation ring~$R$ with
 \[
  k(s')^+ \subseteq R \subseteq k(s')
 \]
 such that the composition
 \[
  \Spec R \longrightarrow \Spec k(s')^+ \overset{c_{s'}}{\longrightarrow} S'^+
 \]
 factors through~$S'$ sending the maximal ideal~$\m_R$ to $\supp(s'_1)$ and such that
 \[
  k(s'_1)^+ = \left(k(s')^+/\m_R\right) \cap k(s'_1).
 \]
 By \cite{EP2005}, Lemma~3.1.5, we find a valuation ring~$R'_1$ of~$k(x')$ extending~$R$ such that
 \[
  k(x')^+ \subseteq R'_1 \subseteq k(x').
 \]
 We consider the solid arrow commutative diagram
 \[
  \begin{tikzcd}
   \Spec k(x')	\ar[rr]	\ar[d]			&					& X	\ar[d]	\\
   \Spec R'_1	\ar[r]	\ar[urr,dashed]	& \Spec R \ar[r]	& S.
  \end{tikzcd}
 \]
 Since $X \to S$ is proper, the dashed arrow exists.
 It defines a horizontal specialization~$x'_1$ of~$x'$ lying over~$s'_1$.
 More precisely, $\supp(x'_1)$ is the image of the closed point of $\Spec R'_1$ in~$X$.
 The valuation ring~$k(x'_1)^+$ is obtained as follows.
 Writing~$\m_{R'_1}$ fo the maximal ideal of~$R'_1$, the quotien $k(x')^+/\m_{R'_1}$ is a valuation ring of $R'_1/\m_{R'_1}$ and we can restrict it to $k(\supp(x'_1)) = k(x'_1)$ (via the dashed arrow):
 \[
  k(x'_1)^+ = \left(k(x')^+/\m_{R'_1}\right) \cap k(x'_1).
 \]
 Finally the map~$c_{x'_1}$ is the composition
 \[
  \Spec k(x'_1)^+ \longrightarrow \Spec k(s'_1)^+ \overset{c_{s'_1}}{\longrightarrow} S^+.
 \]
\end{proof}

\begin{proposition} \label{characterisation_proper_over_S}
 Let $S \to S^+$ be a morphism of schemes and set $\cS = \Spa(S,S^+)$.
 Then the functors
 \[
  \begin{tikzcd}[row sep=tiny]
   \{\text{proper schemes}/S\}	\ar[r,leftrightarrow]	& \{\text{proper adic spaces}/\Spa(S,S^+)\}	\\
   F \colon (X \to S)			\ar[r,mapsto]		& (\Spa(X,S^+) \to \Spa(S,S^+))				\\
   \supp(\cX)					\ar[r,mapsfrom]		& (\cX \to \cS) \colon G
  \end{tikzcd}
 \]
 are mutually inverse equivalences of categories.
\end{proposition}

\begin{proof}
 By \cref{Spa_properties}, proper schemes are indeed mapped to proper adic spaces, so~$F$ is well defined.
 Moreover, the support of a proper adic space over~$\cS$ is proper over~$S$ by \cref{support_properties}~(d), so~$G$ is well defined.
 The composition $G \circ F$ maps a proper scheme~$X$ over~$S$ to the support of $\Spa(X,S)$, which is naturally identified with~$X$.
 Finally, the composite $F \circ G$ maps~$\cX$ over~$\cS$ to $\Spa(\supp(\cX),S^+)$.
 In \cref{vertical_compactification} we have constructed a natural pro-open immersion
 \[
  \cX \hookrightarrow \Spa(\supp(\cX),S^+)
 \]
 with dense image.
 Since $\supp(\cX)$ is separated by \cref{support_properties}, this map is an isomorphism by \cref{immersion_proper} below.
\end{proof}

We formulated the following lemma in more generality than needed here because we are going to use it in the next section, as well.

\begin{lemma} \label{immersion_proper}
 We consider a diagram
 \[
  \begin{tikzcd}
   \cX	\ar[rr,open,"\iota"]	\ar[dr,"f"']	&		& \cX'	\ar[dl,"f'"]	\\
   											& \cS
  \end{tikzcd}
 \]
 of pseudoadic spaces with an open immersion~$\iota$ with dense image.
 If~$f$ is proper and~$f'$ separated, then~$\iota$ is an isomorphism.
\end{lemma}

\begin{proof}
 We consider the base change
 \[
  \cX \times_\cS \cX'.
 \]
The first projection $\cX \times_\cS \cX' \to \cX$ has a section~$s$, whose image $s(\cX)$ is a closed subspace of $\cX \times_\cS \cX'$ since~$f'$ is separated.
The image of $s(\cX)$ in~$\cX'$ is~$\cX$.
By assumption, $\cX \times_\cS \cX' \to \cX'$ is closed, so~$\cX$ is closed in~$\cX'$.
But~$\cX$ is also open with dense image in~$\cX'$.
It follows that the inclusion~$\iota$ is an isomorphism.
\end{proof}

Up until now we have only studied adic spaces but we will have to treat pseudoadic spaces as well.
Proper pseudoadic spaces can be represented by closed subsets of proper adic spaces as we will see in the next proposition.

\begin{proposition} \label{pseudoadic_proper}
 Let $S \to S^+$ be a morphism of qcqs schemes and set $\cS = \Spa(S,S^+)$.
 For a proper pseudoadic space~$\cZ$ over~$\cS$ there exists a proper scheme~$X$ over~$S$ and an open embedding
 \[
  j \colon \underline{\cZ} \hookrightarrow \Spa(X,S^+)
 \]
 such that $j|_{\lvert Z \rvert}$ is a closed embedding.
\end{proposition}

\begin{proof}
 Let~$Z$ be the support of~$\underline{\cZ}$.
 By \cref{vertical_compactification} we have an open immersion
 \[
  \underline{\cZ} \hookrightarrow \Spa(Z,S^+).
 \]
 Let~$X$ be a compactification of~$Z$ over~$S$ (which exists by \cite{Con07} as~$S$ is qcqs).
 Then we consider the composition of open immersions
 \[
  j \colon \underline{\cZ} \hookrightarrow \Spa(Z,S^+) \hookrightarrow \Spa(X,S^+).
 \]
 By \cref{Spa_properties} the adic space $\Spa(X,S^+)$ is proper over~$\cS$.
 Let~$\overline{\cZ}$ be the closure of the image of~$\lvert \cZ \rvert$ in $\Spa(X,S^+)$.
 We obtain an open immersion
 \[
  j_{\lvert \Z \rvert} \colon \lvert \cZ \rvert \hookrightarrow \overline{\cZ}
 \]
 of pseudoadic spaces.
 The source is proper and the target separated over~$\cS$.
 It follows by \cref{immersion_proper} that $\lvert \cZ \rvert \hookrightarrow \overline{\cZ}$ is an isomorphism, or in other words that $j_{\lvert \cZ \rvert} \colon \lvert \cZ \rvert \to \Spa(X,S^+)$ is a closed immersion.
\end{proof}

\section{The proper base change theorem}

 We consider a Cartesian diagram
 \[
  \begin{tikzcd}
   \cX'	\ar[r,"g'"]	\ar[d,"f'"]	& \cX	\ar[d,"f"]	\\
   \cS'	\ar[r,"g"]				& \cS,
  \end{tikzcd}
 \]
 of pseudoadic spaces.
 Then for any sheaf~$\cF$ on the tame site~$\cX_t$ we get a base change morphism
 \[
  g^*Rf_*\cF \overset{\sim}{\longrightarrow} Rf'_*(g'^*\cF).
 \]
 The proper base change theorem is concerned with the question whether it is an isomorphism in case~$f$ is proper.
 By formal reasons it suffices to study the base change homomorphism only in two special settings.
 The first one is the situation where~$\cS$ is tamely local with closed point~$\cS'$.
 In the second case~$\cS$ and~$\cS'$ are tamely closed fields.
 In the second case the base change homomorphism cannot always be an isomorphism (see introduction).
 In this article we only deal with the first case.
 
 The key result for the proper base change theorem is the following.
 
 \begin{proposition} \label{key_bc}
  Let $(R,R^+)$ be a local Huber pair.
  We set $\cS = \Spa(R,R^+)$ and denote by $s \in \cS$ the closed point.
  Let $f \colon X \to \Spec R$ be proper and~$\cF$ a torsion sheaf on the strongly étale site of $\cX = \Spa(X,R^+)$.
  We consider the cartesian diagram
  \[
   \begin{tikzcd}
    \cX_s	\ar[r]	\ar[d,"f_s"']	& \cX	\ar[d,"f"]	\\
    s		\ar[r]					& \cS.
   \end{tikzcd}
  \]
 Then
 \[
  H^i(\cX_\set,\cF) \overset{\sim}{\longrightarrow} H^i(\cX_{s,\set},\cF|_{\cX_s}).
 \]
 \end{proposition}
 
\begin{proof}
 For every $X$-modification~$Y$ of~$R^+$ we consider the commutative diagram
 \begin{equation} \label{X_s_Y_s}
  \begin{tikzcd}
  \cX_s	\ar[r,"\iota_s"]	\ar[d,"\varphi_{Y_c}"']	& \cX	\ar[d,"\varphi_Y"]	\\
  Y_c		\ar[r,"\iota_c"]						& Y,
  \end{tikzcd}
 \end{equation}
 where $c \in \Spec R^+$ denotes the closed point.
 Note that~$Y_c$ equals the image of~$\cX_s$ in~$Y$.
 By \cref{pseudoadic_colim_models} we have functorial isomorphisms
 \[
  H^i(\cX_\set,\cF) \cong \colim_Y H^i(Y_\et,\varphi_{Y,*}\cF)
 \]
 and
 \[
  H^i(\cX_{s,\set},\cF_s) \cong \colim_Y H^i(Y_{c,\et},\varphi_{Y_c,*}\cF_s).
 \]
 But by proper base change for the étale cohomology of a scheme, we get compatible isomorphisms
 \[
  H^i(Y_\et,\varphi_{Y,*}\cF) \overset{\sim}{\longrightarrow} H^i(Y_{c,\et},(\varphi_{Y,*}\cF)_c).
 \]
 It remains to connect the sheaves $(\varphi_{Y,*}\cF)_c = \iota_c^*\varphi_{Y,*}\cF$ and $\varphi_{Y_c,*}\cF_s = \varphi_{Y_c,*} \iota_s^* \cF$.
 At the level of topoi the diagram (\ref{X_s_Y_s}) decomposes as
 \[
  \begin{tikzcd}
   \Shv(\cX_{s,\set})		\ar[r,"\iota_s"]		\ar[d,"\pi_{\cX_s}"]	\ar[dd,bend right=70,"\varphi_{Y_c}"']	& \Shv(\cX_\set)		\ar[d,"\pi_\cX"]	\ar[dd,bend left=70,"\varphi_Y"]	\\
   \Shv(\cX_{s,\set,RZ})	\ar[r,"\iota_{c,RZ}"]	\ar[d,"\Psi_{Y_c}"]												& \Shv(\cX_{\set,RZ})	\ar[d,"\Psi_Y"]	\\
   \Shv(Y_{c,\et})			\ar[r,"\iota_c"]																		& \Shv(Y_\et)
  \end{tikzcd}
 \]
 Accordingly, the base change homomorphism
 \[
  \iota_c^* \varphi_{Y,*} \cF \longrightarrow \varphi_{Y_c,*} \iota_s^* \cF
 \]
 factors as
 \begin{equation} \label{bc_split}
  \iota_c^* \varphi_{Y,*} \cF = \iota_c^* \Psi_{Y,*} \pi_{\cX,*} \cF \longrightarrow \Psi_{Y_c,*} \iota_{s,RZ}^* \pi_{\cX,*} \cF \longrightarrow \Psi_{Y_c,*} \pi_{\cX_s,*} \iota_s^*\cF = \varphi_{Y_c,*} \iota_s^* \cF.
 \end{equation}
 The second homomorphism is an isomorphism by \cref{bc_pi_RZ}.
 The first homomorphism does not have to be an isomorphism.
 This will only turn out to be true in the limit over all $X$-modifications~$Y$ of~$S$.
 
 The equivalence of topoi from \cref{psi_equivalence} entails that the natural homomorphism
 \[
  \cG \longrightarrow \colim_Y \Psi_Y^* \Psi_{Y,*} \cG
 \]
 is an isomorphism for all sheaves~$\cG$ on~$\cX_\set$.
 Similarly,
 \[
  \cG \longrightarrow \colim_Y \Psi_{Y_c}^* \Psi_{Y_c,*} \cG
 \]
 is an isomorphism for all sheaves~$\cG$ on~$\cX_{s,\set}$.
 Using this we identify
 \[
  \colim_Y \Psi_{Y_c}^* \iota_c^* \Psi_{Y,*} \pi_{X,*} \cF = \colim_Y \iota_{s,RZ}^* \Psi_Y^* \Psi_{Y,*} \pi_{X,*} \cF \longrightarrow \colim_Y \Psi_{Y_c}^* \Psi_{Y_c,*} \iota_{s,RZ}^* \pi_{X,*}\cF
 \]
 with the identity
 \[
  \iota_{s,RZ}^* \pi_{X,*} \cF \longrightarrow \iota_{s,RZ}^* \pi_{X,*} \cF.
 \]
 Consequently,
 \begin{align*}
  \colim_Y H^i(Y_{c,\et},\iota_c^*\Psi_{Y,*}\pi_{X,*}\cF)	&= H^i(\cX_{x,RZ,\set},\colim_Y \Psi_{Y_c}^*\iota_c^*\Psi_{Y,*}\pi_{X,*}\cF)	\\
															&= H^i(\cX_{s,RZ,\set},\colim_Y \Psi_{Y_c}^* \Psi_{Y_c,*}\iota_{x,RZ}^*\pi_{X,*}\cF)	\\
													&= \colim_Y H^i(Y_{c,\et},\Psi_{Y_c,*} \iota_{s,RZ}^* \pi_{X,*}\cF).
 \end{align*}
 Via the second base change map in (\ref{bc_split}), which we already showed to be an isomorphism, we identify this with
 \[
  \colim_Y H^i(Y_{c,\et},\Psi_{Y_c,*} \iota_s^* \cF)
 \]
 and we are done
 \end{proof}

We want to enhance \cref{key_bc} to a proper base change theorem along locally closed immersions.
We first need to settle the baby example where the proper morphism is a closed immersion.

\begin{lemma} \label{closed_immersion_acyclic}
 Let $\iota \colon \cZ \to \cX$ be a closed immersion of pseudoadic spaces and~$\cF$ an abelian sheaf on~$\cZ_\set$.
 Then
 \[
  R\iota_*\cF \cong \iota_*\cF.
 \]
\end{lemma}

\begin{proof}
 We have to show that $R^i\iota_*\cF = 0$ for $i>0$.
 We can check this on stalks.
 We now have to show the following statement.
 Let~$\cX$ be a strongly local pseudoadic space and $\cZ \subseteq \cX$ a closed subspace.
 For any abelian sheaf~$\cF$ on~$\cZ_\set$
 \[
  H^i(\cZ_\set,\cF) = 0 \qquad \forall i > 0.
 \]
 But closed subspaces of strongly local spaces are strongly local, so these cohomology groups vanish.
\end{proof}

\begin{lemma} \label{bc_closed_immersion}
 Let $\iota \colon \cZ \hookrightarrow \cX$ be a closed immersion of pseudoadic spaces and~$\cF$ a sheaf on~$\cX_\set$
 Then for any cartesian square
 \[
  \begin{tikzcd}
   \cZ'	\ar[r,"g'"]	\ar[d,closed,"\iota'"']	& \cZ	\ar[d,closed,"\iota"]	\\
   \cX'	\ar[r,"g"']							& \cX
  \end{tikzcd}
 \]
 the base change homomorphism
 \[
  g^*\iota_*\cF \longrightarrow \iota'_* g'^* \cF
 \]
 is an isomorphism (note that $\iota_* = R\iota_*$ by \cref{closed_immersion_acyclic}).
\end{lemma}

\begin{proof}
 We can check this on stalks and therefore assume that~$\cX$ and~$\cX'$ are strongly henselian.
 But then also~$\cZ$ and~$\cZ'$ are strongly henselian.
 By construction, $(g'^*\cF)(\cZ')$ equals $\colim_{\cU \to \cZ} \cF(\cU)$, where the colimit runs over all strongly étale morphisms $\cU \to \cZ$ with a factorization $\cZ' \to \cU \to \cZ$.
 Every such $\cU \to \cZ$ splits.
 We conclude that $g'^*\cF(\cZ') = \cF(\cZ)$, which translates to the statement of the lemma.
\end{proof}

Suppose $\cZ \hookrightarrow \cX$ is a closed immersion of adic spaces.
We denote by $\lvert \cZ \rvert$ the image of~$\cZ$ in~$\cX$.
The pseudoadic space $(\cX, \lvert \cZ \rvert)$ should in a suitable sense be the same as~$\cZ$.
At the level of topoi this is true by the following lemma.

\begin{lemma} \label{topoi_pseudoadic}
 Let $\iota \colon \cZ \to \cX$ be a locally closed immersion of pseudoadic spaces inducing a homeomorphism $\lvert \cZ \rvert \to \lvert \cX \rvert$.
 Then for $\tau \in \{\set,t,\et\}$ the associated morphism of sites $\cZ_\tau \to \cX_\tau$ is an equivalence.
\end{lemma}

\begin{proof}
 The statement for the étale site is true by \cite{Hu96}, Corollary~2.3.8.
 Since tameness and strong étaleness is only meassured over points of~$\lvert \cZ \rvert$ and~$\lvert \cX \rvert$, respectively, and $\lvert \cZ \rvert \to \lvert \cX \rvert$ is a homeomorphism with trivial residue field extensions, the tameness and unramifiedness conditions on either side match.
 Therefore, $\cZ_\tau \to \cX_\tau$ is an equivalence.
\end{proof}

From \cref{key_bc} we deduce the proper base change theorem for immersions in the strongly étale topology.

\begin{theorem} \label{bc_set}
 Let $f \colon \cX \to \cS$ be a proper morphism of discretely ringed pseudoadic spaces and $\iota \colon \cS' \to \cS$ a locally closed immersion.
 We set $\cX' := \cX \times_\cS \cS'$ thus obtaining a cartesian square
 \[
  \begin{tikzcd}
   \cX'	\ar[r,"\iota'"]	\ar[d,"f'"']	& \cX	\ar[d,"f"]	\\
   \cS'	\ar[r,"\iota"']					& \cS.
  \end{tikzcd}
 \]
 Then for any $i \in \N$ and any sheaf of abelian groups, the base change homomorphism
 \[
  \iota^* Rf_* \cF \longrightarrow Rf'_* \iota'^* \cF
 \]
 is an isomorphism.
\end{theorem}

\begin{proof}
 We can check on stalks that the base change homomorphism is an isomorphism.
 This reduces us to the case where~$\cS$ and~$\cS'$ are strongly henselian and $\cS' \to \cS$ is local.
 Since $\cS' \to \cS$ is also a locally closed immersion, it is a closed immersion in our situation.
 It suffices to prove the theorem in case~$\cS'$ is the closed point~$s$ of~$\cS$.
 The general case then follows by considering the composition
 \[
  s \hookrightarrow \cS' \hookrightarrow \cS
 \]
 and using the result for $s \hookrightarrow \cS'$ and $s \hookrightarrow \cS$.
 Here we need to be precise and specify what we mean by the closed point.
 The smallest possible closed subspace of~$\cS$ containing the closed point is $(\Spa(k(s),k(s)^+),s)$.
 We can certainly work with this space.
 Morally, it should coincide with $(\underline{\cS},s)$.
 By \cref{topoi_pseudoadic} the closed immersion $(\Spa(k(s),k(s)^+,s) \to (\underline{\cS},s)$ induces an equivalence of strongly étale topoi.
 We conclude that it suffices to prove the theorem for $\cS' = (\underline{\cS},s)$.
 In the following we will just write~$s$ for this space.
 
 The pseudoadic space~$\cS$ is of the form $(\Spa(R,R^+),\lvert \cS \rvert)$ for a strongly henselian Huber pair $(R,R^+)$ and a closed subset $\lvert \cS \rvert \subseteq \Spa(R,R^+)$.
 Using \cref{pseudoadic_proper} we find a closed embedding of~$\cX$ into a adic space of the form $\Spa(X,R^+)$ for a proper scheme~$X$ over~$R^+$.
 Denoting it by
 \[
  h \colon \cX \hookrightarrow \Spa(X,R^+)
 \]
 we obtain proper base change for $\cX \to \cS$ and the sheaf~$\cF$ on~$\cX_\set$ by applying proper base change for $\Spa(X,R^+) \to \Spa(R,R^+)$ to the sheaf $h_*\cF$ and using that~$h_*$ is exact by \cref{closed_immersion_acyclic} and satisfies base change by \cref{bc_closed_immersion}.
 But for the square
 \[
  \begin{tikzcd}
   \Spa(X,R^+)_s	\ar[r]	\ar[d]	& \Spa(X,R^+)	\ar[d]	\\
   s				\ar[r]			& \Spa(R,R^+)
  \end{tikzcd}
 \]
 the base change homomorphism is an isomorphism by \cref{key_bc}.
\end{proof}

Our next goal is to deduce from \cref{bc_set} the corresponding proper base change statement for the tame topology.

\begin{lemma} \label{bc_t_set}
 Let~$\cX$ be a pseudoadic space of specialization characteristic $\ch^+(\cX) = p > 0$.
 For a locally closed immersion $f \colon \cY \to \cX$ of pseudoadic spaces we consider the diagram
 \[
  \begin{tikzcd}
   \cY_t	\ar[r,"f_t"]		\ar[d,"\nu_\cY"']	& \cX_t	\ar[d,"\nu_\cX"]	\\
   \cY_\set	\ar[r,"f_\set"']						& \cX_\set.
  \end{tikzcd}
 \]
 For any $p$-torsion sheaf~$\cF$ on~$\cX_t$ the base change homomorphism
 \[
  f_\set^* R\nu_{\cX,*} \cF \longrightarrow R\nu_{\cY,*} f_t^*\cF
 \]
 is an isomorphism.
\end{lemma}

\begin{remark}
 Note that we know from \cite{HueAd}, Proposition~8.5 that $R\nu_{\cX,*} \cF = \nu_{\cX,*} \cF$ and similarly for~$\nu_\cY$.
\end{remark}

\begin{proof}
 Looking at stalks we reduce to the situation where~$\cX$ and~$\cY$ are strongly local and $f \colon \cY \to \cX$ is a closed immersion.
 Now we need to show that the restriction $\cF(\cX) \to f^*\cF(\cY)$ is an isomorphism.
 But both sides identify with the invariants~$\cF_{\bar{x}}^G$, where~$\bar{x}$ is a geometric point above the closed point $x \in \cX$ and~$G$ is the Galois group of the maximal tame extension of $k(x)$ (inside $k(\bar{x})$).
\end{proof}

\begin{corollary} \label{bc_tame_p}
 We consider the cartesian square
 \[
  \begin{tikzcd}
   \cX'	\ar[r,"\iota'"]	\ar[d,"f'"']	& \cX	\ar[d,"f"]	\\
   \cS'	\ar[r,"\iota"']					& \cS
  \end{tikzcd}
 \]
 of discretely ringed adic spaces where $f \colon \cX \to \cS$ is proper and $\iota \colon \cS' \to \cS$ is a locally closed immersion.
 Let $p > 0$ be the specialization characteristic $\ch^+(\cS)$.
 Then for any $p$-torsion sheaf~$\cF$ on~$\cX_t$ the base change homomorphism
 \[
  \iota^* Rf_* \cF \longrightarrow R f'_* \iota'^* \cF
 \]
 is an isomorphism for all $i \in \N$.
\end{corollary}

\begin{proof}
 Looking at stalks we reduce to~$\cS$ and~$\cS'$ being tamely local.
 We now need to show that
 \[
  H^i(\cX,\cF) \cong H^i(\cX',\iota'^*\cF).
 \]
 We consider the morphism of sites
 \[
  \nu_\cX \colon \cX_t \longrightarrow \cX_\set
 \]
 and similarly for~$\cS$, $\cS'$, and~$\cX'$.
 By \cite{HueAd}, Proposition~8.5 we know
 \[
  R\nu_{\cX,*} \cF = \nu_{\cX,*} \cF
 \]
 for every $p$-torsion sheaf~$\cF$ on~$\cX_t$.
 Using \cref{bc_t_set} we get a natural identification
 \begin{align*}
  H^i(\cX'_t,\iota'^*\cF)	& \cong H^i(\cX'_\set,\nu_{\cX',*}\iota'^*\cF)	\\
							& \cong H^i(\cX'_\set,\iota'^*\nu_{\cX,*}\cF).
 \end{align*}
 By \cref{bc_set} we get
 \begin{align*}
  H^i(\cX'_\set,\iota'^*\nu_{\cX,*}\cF)	& \cong H^i(\cX_\set,\nu_{\cX,*}\cF)	\\
										& H^i(\cX_t,\cF)
 \end{align*}
 and we are done.
\end{proof}

\begin{remark}
 In \cref{bc_tame_p} we only talked about $p$-torsion sheaves.
 One might think that proper base change for torsion sheaves whose torsion is invertible on the adic space~$\cX$ should already be known.
 This is not completely the case.
 It is true that the comparison of tame with étale cohomology (\cite{HueAd}, Proposition~8.2) reduces us to proving proper base change for the étale topology.
 However, the étale proper base change theorem has only been established for \emph{analytic} adic spaces (see \cite{Hu96}, Theorem~4.4.1) and here we are dealing with the opposite extreme: discretely ringed adic spaces.
 One might hope to solve the problem by recurring to the proper base change theorem for schemes.
 But then we cannot treat all torsion sheaves but only those that are algebraic, meaning that they come from the algebraic tame site.
 We will address this task in subsequent work.
\end{remark}

\bibliographystyle{../meinStil}
\bibliography{../citations}

\end{document}